\documentclass[11pt]{amsart}

\usepackage[mathscr]{eucal}
\usepackage{amsmath,amssymb,amsfonts,amsthm,enumerate}
\usepackage{hyperref}

\newtheorem{theorem}{Theorem}[section]
\newtheorem{lemma}[theorem]{Lemma}

\newtheorem{proposition}[theorem]{Proposition}

\newcommand{\N}{\mathbb N}
\newcommand{\Z}{\mathbb Z}
\newcommand{\R}{\mathbb R}
\newcommand{\Q}{\mathbb Q}

\DeclareMathOperator{\ord}{ord}

\DeclareMathOperator{\supp}{Supp}

\newcommand{\la}{\langle}
\newcommand{\ra}{\rangle}
\newcommand{\be}{\begin{equation}}
\newcommand{\ee}{\end{equation}}
\newcommand{\und}{\;\mbox{ and }\;}
\newcommand{\nn}{\nonumber}
\newcommand{\ber}{\begin{eqnarray}}
\newcommand{\eer}{\end{eqnarray}}
\newcommand{\Sum}[2]{\underset{#1}{\overset{#2}{\sum}}}
\newcommand{\Summ}[1]{\underset{#1}{\sum}}

\newcommand{\vp}{\mathsf v}

\DeclareSymbolFont{goo}{OMS}{cmsy}{b}{n}
\DeclareMathSymbol{\gooT}{\mathalpha}{goo}{"1}
\newcommand{\bdot}{\mathbin{\gooT}}
\begin{document}

\title[A Generalization of the  Chevalley-Warning and Ax-Katz Theorems]{
A Generalization of the  Chevalley-Warning and Ax-Katz Theorems with a View Towards Combinatorial Number Theory}

\author{David J. Grynkiewicz}
\address{Department of Mathematical Sciences\\ University of Memphis\\ Memphis, TN 38152\\
USA}
\email{diambri@hotmail.com}

\subjclass[2020]{11T06, 11B75, 20D60, 11G25}
\keywords{Chevalley-Warning Theorem, Ax-Katz Theorem, Zero-Sum, Erd\H{o}s-Ginzburg-Ziv, Davenport Constant}

\begin{abstract}Let $\mathbb F_q$ be a finite field of characteristic $p$ and order $q$. The Chevalley-Warning Theorem is a classical result which asserts that the set $V$ of common zeros of a collection of polynomials  must satisfy $|V|\equiv 0\mod p$, provided the number of variables is sufficiently large with respect to the degrees of the polynomials. The Ax-Katz Theorem generalizes this by giving tight bounds for higher order $p$-divisibility for $|V|$. Besides the intrinsic algebraic interest of these results, they are also important tools in the Polynomial Method, particularly in the prime field  case $\mathbb F_p$, where they have been used to prove many results in Combinatorial Number Theory.
In this paper, we begin by explaining how arguments used by R. Wilson to give an elementary proof of the $\mathbb F_p$ case for the Ax-Katz Theorem can also be used to prove the following generalization of the Chevalley-Warning and Ax-Katz Theorems for $\mathbb F_p$, where we allow varying prime power moduli.  Given any box $\mathcal B=\mathcal I_1\times\ldots\times\mathcal I_n$, with each $\mathcal I_j\subseteq\Z$ a complete system of residues modulo $p$, and a collection of nonzero polynomials $f_1,\ldots,f_s\in \mathbb Z[X_1,\ldots,X_n]$, then the set of common zeros inside the box,
$$V=\{\mathbf a\in \mathcal B:\; f_1(\textbf a)\equiv 0\mod p^{m_1},\ldots,f_s(\textbf a)\equiv 0\mod p^{m_s}\},$$ satisfies $|V|\equiv 0\mod p^m$, provided $n>(m-1)\max_{i\in [1,s]}\Big\{p^{m_i-1}\deg f_i\Big\}+ \Sum{i=1}{s}\frac{p^{m_i}-1}{p-1}\deg f_i.$
The introduction of the box $\mathcal B$ adds a degree of flexibility, in comparison to prior work of Zhi-Wei Sun. Indeed, incorporating the ideas of Sun, a weighted version of the above result is given.
We continue by explaining how the added flexibility, combined with an appropriate use of Hensel's Lemma to choose the complete system of residues $\mathcal I_j$, effectively allows many  combinatorial applications of the Chevalley-Warning and Ax-Katz Theorems, previously only valid for $\mathbb F_p^n$, to extend with bare minimal modification to validity for an arbitrary finite abelian $p$-group $G$.  We illustrate this be giving several examples, including a new proof of the exact value of the Davenport Constant $\mathsf D(G)$ for finite abelian $p$-groups, and a streamlined proof of the Kemnitz Conjecture. We also derive some new results, for a finite abelian $p$-group $G$ with exponent $q$, regarding the constant $\mathsf s_{kq}(G)$, defined as the minimal integer $\ell$ such that any sequence of $\ell$ terms from $G$ must contain a zero-sum subsequence of length $kq$. Among other results for this constant, we show that $\mathsf s_{kq}(G)\leq kq+\mathsf D(G)-1$ provided $k>\frac{d(d-1)}{2}$ and $p> d(d-1)$, where $d=\left \lceil \frac{\mathsf D(G)}{q}\right\rceil$, answering  a problem of   Xiaoyu He in the affirmative by removing all dependance on $p$ from the bound for $k$.
\end{abstract}

\maketitle

\section{Introduction and Notation}

\subsection{Basic Notation}
Let $\mathbb N_0=\{0,1,2,\ldots\}$ and $\mathbb N=\{1,2,\ldots,\}$, and  let $\mathbb F_q$ denote the finite field of order $q$, whose characteristic must then be a prime $p\geq 2$ with $q$ a power of $p$. For a commutative ring $R$, we let $R[X_1,\ldots,X_n]$ denote the polynomial ring in the variables $X_1,\ldots,X_n$ with coefficients from $R$, and we often use $\textbf x=(X_1,\ldots,X_n)$ to denote the tuple of variable inputs.  Each $f\in R[X_1,\ldots,X_n]$ is then a finite sum of \emph{monomials} $f(\textbf x)=\Summ{(k_1,\ldots,k_n)\in \mathbb N_0^n} c_{k_1,\ldots,k_n}X_1^{k_1}\cdots X_n^{k_n}$ with coefficients $c_{k_1,\ldots,k_n}\in R$. The monomials that occur in $f$ are then the summands with $c_{k_1,\ldots,k_n}\neq 0$.  The \emph{degree} of $f$ is denoted $\deg f$ and is the maximal value of $k_1+\ldots+k_n$ as we range over all tuples $(k_1,\ldots,k_n)\in \mathbb N_0^n$ with $c_{k_1,\ldots,k_n}\neq 0$. The zero-polynomial $f=0$ has $\deg f=-1$ by convention. For $j\in[1,n]$, we use $\deg_j f$ to denote the degree of $f$ in the $j$-th variable $X_j$. Throughout the paper, the expression $0^0:=1$, being interpreted as the constant polynomial $X^0=1$ evaluated at $0$. A polynomial $f\in \mathbb Q[X_1,\ldots,X_n]$ is called an \emph{integer valued polynomial} if $f(\textbf a)\in \Z$ for all $\textbf a\in \Z^n$. We use $\mathsf{Int}(\Z)$ to denote the set of all integer valued polynomials $f\in \Q[X]$.
An integer valued map $f:\Z\rightarrow \Z$ is \emph{periodic }with period $m$ if $f(x+m)=f(x)$ for all $x\in \Z$. All intervals are discrete, so $[a,b]=\{x\in\Z:a\leq x\leq b\}$ for $a,\,b\in\R$, and variables introduced with an inequality are assumed to be integers. Given an integer $m\geq 1$, a \emph{complete system of residues modulo $m$} is a set $\mathcal I\subseteq \Z$ with $|\mathcal I|=m$ whose elements are distinct modulo $m$, i.e., $\mathcal I$  contains exactly one representative for every residue class modulo $m$. We use $\varphi$ to denote the Euler totient function, so $\varphi(n)$ is the number of elements $x\in [1,n]$ that are relatively prime to the integer $n\geq 1$. In particular, $$\varphi(1)=1\quad\und\quad \varphi(q)=\frac{(p-1)q}{p}$$ for a prime power $q=p^s>1$. Given a prime $p\geq 2$ and $x\in \Z$, we let $\vp_p(x)$ denote the $p$-adic valuation of $x$, which is simply the multiplicity of the prime $p$ in the prime factorization of $x$, and we extend this for $x=\frac{a}{b}\in \Q$ with $a,\,b\in \Z$ by the standard definition $\vp_p(x)=\vp_p(a)-\vp_p(b)$. For an element $X$ in a commutative ring containing $\Q$, the binomial coefficient is defined as $$\binom{X}{n}=\frac{X(X-1)\cdots(X-n+1)}{n!},$$ with $\binom{X}{0}:=1$. If $x\in \mathbb N_0$ is an integer, then $\binom{x}{n}$ counts the number of ways to choose $n$ elements from a set of size $x$, and is thus an integer. Moreover, $\binom{x}{n}=0$ for $x\in\mathbb N_0$ and  $n>x$.

\subsection{Introduction}

The study of the  common roots of a collection of polynomials $f_1,\ldots,f_s\in R[X_1,\ldots,X_n]$ is a classical object of study in Commutative Algebra. When $R=\mathbb F_q$ is a finite field of characteristic $p$, one of the most well-known such results is the Chevalley-Warning Theorem \cite{G-book} \cite{nat-book} \cite{Tao-Vu-book}.

\begin{theorem}[Chevalley-Warning Theorem (1936)]\label{thm-chev-warning} Let $\mathbb F_q$ be a finite field of characteristic $p$, let $f_1,\ldots,f_s\in \mathbb F_q[X_1,\ldots,X_n]$ be nonzero polynomials, where $s\geq 1$,  and let $$V=\{\textbf a\in \mathbb F_q^n:\; f_1(\textbf a)=0,\ldots,f_s(\textbf a)=0\}.$$ If $n>\Sum{i=1}{s}\deg f_i$, then $|V|\equiv 0\mod p$.
\end{theorem}

As a particular case, if there is one common zero for the polynomials $f_1,\ldots,f_s$, then there must be at least one nontrivial zero, which was the original result of Chevalley \cite{chevalley}, afterwards generalized by Warning \cite{warning}.  Later, the higher order $p$-divisibility of $|V|$ was considered by Ax \cite{Ax} (for $s=1$) and then  for general $s$ by Katz \cite{Katz}, resulting in what is known as the Ax-Katz Theorem.

\begin{theorem}[Ax-Katz Theorem (1971)]\label{thm-ax-katz} Let $\mathbb F_q$ be a finite field of characteristic $p$, let $f_1,\ldots,f_s\in \mathbb F_q[X_1,\ldots,X_n]$ be nonzero polynomials, where $s\geq 1$,  and let $$V=\{\textbf a\in \mathbb F_q^n:\; f_1(\textbf a)=0,\ldots,f_s(\textbf a)=0\}.$$ If $n>(m-1)\max_{i\in [1,s]}\{\deg f_i\}+\Sum{i=1}{s}\deg f_i$, where $m\geq 1$, then $|V|\equiv 0\mod p^m$.
\end{theorem}

Both the Chevalley-Warning and Ax-Katz Theorems have attracted considerable attention in Commutative Algebra, including many extensions, refinements, variants and alternative proofs. See \cite{Aichinger-chewarning} \cite{brink-chevwarning-restricted}  \cite{clark-gen-che-warning} \cite{cao-ax-katz-gen} \cite{cao-sun-ax-katz-variant} \cite{castro-ax-katz-var} \cite{clark-boundary-chev-warning} \cite{clark-warning-restr} \cite{Hou-reduction-arg-ax-katz} \cite{moreno-moreno} \cite{moreno-shum-chev-warinng-var} \cite{Wan-altproof-Ax-Katz} \cite{wan-altproof-Ax-katzII} \cite{Wilson-ax-katz} for a handful of such instances among many more. However, the interest in these results extends much further, also   to areas such as Discrete Mathematics, where they form a standard tool in the ``Polynomial Method.'' Here, the interest lies not directly in the results themselves but rather in  what  other results can be proved by their usage in combination with appropriately chosen polynomials.
For such reasons, the Chevalley-Warning Theorem is often found in many texts on Additive Combinatorics, e.g. \cite{G-book} \cite{nat-book} \cite{Tao-Vu-book}, and is an indispensable tool in many parts of Combinatorics. As this will be a prime focus in this paper, we will shortly see concrete examples of how this works. Worth noting, regarding the use of the Ax-Katz and Chevalley-Warning Theorems in Discrete Mathematics, the case $\mathbb F_p$ is the main
focus of interest, and thus for this paper as well.

Despite the rather elementary formulation of the Ax-Katz Theorem, most proofs are rather non-elementary, to varying extents. Perhaps the most elementary proof, though only valid for  $\mathbb F_p$, was given by Wilson \cite{Wilson-ax-katz}. His interest was primarily in using the method he developed to give striking applications in Coding Theory, and while his work received some attention in Coding Theory, it's importance outside Coding Theory seems not  fully realized. The first part of this paper is devoted to detailing how the method of Wilson readily adapts to prove the following generalization of the Ax-Katz and Chevalley-Warning Theorems, where we are allowed to consider polynomial equations modulo varying prime powers $p^{m_i}$.

\begin{theorem}\label{thm-axkatz-gen} Let $p\geq 2$ be prime, let $n\geq 1$ and $\mathcal B=\mathcal I_1\times \ldots\times \mathcal I_n$ with each $\mathcal I_j\subseteq \Z$ for $j\in[1,n]$  a complete system of residues  modulo  $p$, let $s\geq 1$ and  $m_1,\ldots,m_s\geq 0$ be integers, let $f_1,\ldots, f_s\in \mathbb \Z[X_1,\ldots,X_n]$ be nonzero polynomials, let $w_1(X),\ldots,w_s(X)\in\Q[X]$ be integer valued polynomials with respective degrees $t_1,\ldots,t_s\geq 0$, and let \begin{align*}V=\{\textbf a\in \mathcal B:\; f_i(\textbf a)\equiv 0\mod p^{m_i}\mbox{ for all $i\in [1,s]$}\}\quad\und \quad N=\Summ{\textbf a\in V}\prod_{i=1}^{s}w_i\Big(
\frac{f_i(\textbf a)}{p^{m_i}}\Big).\end{align*}
 If $n>(m-1)\max_{i\in [1,s]}\Big\{1,\;\frac{\varphi(p^{m_i})}{p-1}\deg f_i\Big\}+ \Sum{i=1}{s}\frac{(t_i+1)p^{m_i}-1}{p-1}\deg f_i,$ where $m\geq 1$ and $\varphi$ denotes the Euler totient function, then
$$N\equiv 0\mod p^m.$$
\end{theorem}

In the special case in Theorem \ref{thm-axkatz-gen} when all $w_i=1$ are constant polynomials, we find that $N=|V|$ is simply the cardinality of $V$, with $t_i=0$ for all $i$. Additionally assuming $m_i=1$ for all $i$, we then recover the Ax-Katz Theorem for $\mathbb F_p$.
In general, the quantity $N$ counts the elements $\textbf a\in V$ each with  multiplicity $w_i\left(\frac{f_i(\textbf a)}{p^{m_i}}\right)$, meaning $N$ may be view as the weighted size of $V$ using the integer valued polynomials $w_1,\ldots,w_s\in \Q[X]$ as weight functions.
The idea to consider such weight functions is due to Zhi-Wei Sun \cite{Sun-Ax-Katz-Wilson}, who indeed noticed (in his unpublished preprint from 2006) that Wilson's argument could be used to prove a result of the form stated in Theorem \ref{thm-axkatz-gen}, specifically, in the case $\mathcal I_j=[0,p-1]$ for all $j$.  However, as already alluded to, we are primarily interested in the \emph{application} of Theorem \ref{thm-axkatz-gen}, particularly to Combinatorial Number Theory, and for this, the added flexibility gained by considering common zeros inside the box $\mathcal B=\mathcal I_1\times \ldots\times \mathcal I_n$, with the $\mathcal I_j$ allowed to be \emph{any} complete system of residues modulo $p$,  will be quite crucial.
This will become clearer once we have some examples, but the crux of the matter is that, by choosing the $\mathcal I_j$ carefully, we can simulate behavior modulo $p^{m}$ that could normally only be expected modulo $p$, at least so long as we restrict to elements $x\in \mathcal I_j$.

For instance, Fermat's Little Theorem \cite{niven-zuck-mont-nr-text} tells us that
$$x^{p-1}\equiv
\left\{
  \begin{array}{ll}
    1 \mod p& \hbox{if $x\not\equiv 0\mod p$} \\
    0\mod p & \hbox{if $x\equiv 0\mod p$.}
  \end{array}
\right.$$ From a combinatorial point of view, this is quite nice, as it tells us that the polynomial $X^{p-1}$ can be used as an indicator function modulo $p$. Indeed, in many applications of the Chevalley-Warning or Ax-Katz Theorem in Combinatorial Number Theory, this is the key means of translating between combinatorial information and the  algebraic information gleamed from the Chevalley-Warning or Ax-Katz Theorem. Fermat's Little Theorem, of course, fails modulo higher powers of $p$. Nonetheless, Hensel's Lemma can be used  to find an appropriate $\mathcal I_j$ for which Fermat's Little Theorem   holds modulo $p^m$,  when restricted to $x\in \mathcal I_j$. We include the short derivation of Proposition \ref{prop-Ax-Katz-ChooseI} at the end of Section \ref{sec-alg}.

\begin{proposition}\label{prop-Ax-Katz-ChooseI}
Let $p\geq 2$ be  prime and let $m\geq 1$. There exists a complete system of residues $\mathcal I\subseteq [0,p^m-1]$  modulo $p$ such that $$x^{p-1}\equiv
\left\{
  \begin{array}{ll}
    1 \mod p^m& \hbox{if $x\not\equiv 0\mod p$} \\
    0\mod p^m & \hbox{if $x\equiv 0\mod p$,}
  \end{array}
\right.\quad\mbox{ for every $x\in \mathcal I$}.$$
\end{proposition}

The main point is that, using  Proposition \ref{prop-Ax-Katz-ChooseI} (or a more general application of Hensel's Lemma) to choose the $\mathcal I_j$ appropriately, it is then often possible to use
Theorem \ref{thm-axkatz-gen}, in place of either the Chevalley-Warning or Ax-Katz Theorem, and de-facto obtain a result via the Polynomial Method for a general finite abelian $p$-group $G$ that could previously only be achieved by the same means for the special case $G=\mathbb F_p^r$.
It is this point that we wish to particularly highlight, and for which we provide several examples  illustrating the idea.

The first example   regards the Davenport Constant $\mathsf D(G)$ of a finite abelian group $G$, defined as the minimal integer $\ell$ such that every sequence of $\ell$ terms from $G$ must contain a nontrivial subsequence whose terms sum to zero (called a zero-sum subsequence). It is an invariant that has received considerable attention, in part due to its connection with Algebraic Number Theory. It is perhaps best simply to refer to the texts \cite{alfred-book} \cite{G-book}, and the many references therein, for broader context. In general, if $G=(\Z/n_1\Z)\times\ldots\times (\Z/n_r\Z)$ with $n_1\mid\ldots \mid n_r$, then a rather simple construction shows that $$\mathsf D(G)\geq \mathsf D^*(G):=1+\Sum{i=1}{r}(n_i-1).$$ While even the (near) exact determination of $\mathsf D(G)$ remains an important and challenging question for a general finite abelian group $G$, the following classical result of Olson \cite{olson-dav-pgroups} and also van Emde Boas and Kruyswijk \cite{deboas-dav-pgroup} showed that the trivial lower bound is tight for $p$-groups. Both these original proofs relied upon ideals and group algebras. Our first application  will be to use Theorem \ref{thm-axkatz-gen} to give a fairly direct proof of Theorem \ref{thm-dav-pgroup}.

\begin{theorem}
\label{thm-dav-pgroup} Let $G$ be a finite abelian $p$-group. Then $$\mathsf D(G)=\mathsf D^*(G).$$
\end{theorem}

The next example   regards the Erd\H{o}s-Ginzburg-Ziv Constant $\mathsf s(G)$ of the finite abelian group $G$, defined as the minimal integer $\ell$ such that every sequence of $\ell$ terms from $G$ must contain a zero-sum subsequence of length $\exp(G)$ (the exponent of $G$). The Erd\H{o}s-Ginzburg-Ziv Theorem implies that $\mathsf s(\Z/n\Z)=2n-1$ \cite{G-book} \cite{nat-book} \cite{alfred-book}. It was an open conjecture of Kemnitz \cite{kemnitz-originalpaper} that $\mathsf s((\Z/n\Z)^2)=4n-3$, for which a simple argument shows that it suffices to consider the case $n=p$ prime. Partial progress towards this conjecture was achieved by Alon and Dubiner \cite{alon-dubiner-org} and by R\'onyai \cite{ronyai-kemnitz} before finally being resolved by Reiher \cite{reiher-kemnitz-conj} (and also di Fiore \cite{sav-chev-kemnitz-fiore}).  Regarding higher rank groups $(\Z/n\Z)^r$, Alon and Dubiner  gave a linear bound via Algebraic Graph Theory \cite{alon-dub-linbound}. Reiher's proof involved combining the Chevalley-Warning Theorem with several combinatorial double counting arguments. Ronyai's proof was also algebraic, but instead made use of linear algebra surrounding multi-linear monomials.
Our second application  will be to use Theorem \ref{thm-axkatz-gen} to give a streamlined proof of Theorem \ref{thm-kemnitzconj}. As we will see, the flexibility of being able to use more general weights  allows us to directly derive some of the congruences used in Reiher's proof, reducing the number of   ad-hoc combinatorial doubling counting arguments needed.
This is not surprising since the proof of the weighted Weisman-Fleck congruence \cite{Sun-Wan-elem-weight-weis-fleck}, which is one of the key components used in the proof of Theorem \ref{thm-axkatz-gen}, already incorporates such double counting arguments into its proof, meaning they are in some sense built into  Theorem \ref{thm-axkatz-gen} itself.  While the proof of Theorem \ref{thm-kemnitzconj} is only a minor variation on Reiher's, it does highlight how the weight functions can be used to generate additional linearly independent congruences in a routine manner. For more complicated arguments, this can simplify the technical calculations and help focus attention on the more involved parts of the argument.

\begin{theorem}[Kemnitz Conjecture]\label{thm-kemnitzconj}
Let $C_p$ be a cyclic group of order $p\geq 2$ prime. Then $$\mathsf s(C_p^2)=4p-3.$$
\end{theorem}

The  final examples   regard a generalized Erd\H{o}s-Ginzburg-Ziv constant $\mathsf s_{k\exp(G)}(G)$ of the finite abelian group $G$, defined as the minimal integer $\ell$ such that every sequence of $\ell$ terms from $G$ must contain a zero-sum subsequence of length $k\exp(G)$.
See  \cite{bitz-skq-lowerbounds} \cite{Gao-skq-basicell} \cite{gao-han-skq-specialcase-var} \cite{Gao-inverse-skq} \cite{Gao-thanga-skq-lowk} \cite{han-skq-rank4} \cite{he-skq} \cite{kubertin-skq} for some relevant examples of results regarding $\mathsf s_{k\exp(G)}(G)$.
More generally, given a subset $X\subseteq \N_0$, we let
$\mathsf s_{X}(G)$ be the minimal integer $\ell$ such that every sequence of $\ell$ terms from $G$ must contain a zero-sum subsequence $T$ with length $|T|\in X$.
Here, we will particularly focus on a question initially raised by Kubertin \cite{kubertin-skq} and later extended  in  \cite{gao-han-skq-specialcase-var}. The problem, for a finite abelian  group $G$, is to find an optimal bound $\ell(G)$ such that  $\mathsf s_{k\exp(G)}(G)=k\exp(G)+\mathsf D(G)-1$ for all $k\geq \ell(G)$.
The corresponding lower bound for $s_{k\exp(G)}(G)$ follows from a rather basic construction, so the issue is how large must $k$ be to ensure  $\mathsf s_{k\exp(G)}(G)\leq k\exp(G)+\mathsf D(G)-1$.
An older result of Gao implies this is true for $k\geq \frac{|G|}{\exp(G)}$ \cite{Gao-skq-basicell}, and it was conjectured in  \cite{kubertin-skq}   \cite{gao-han-skq-specialcase-var} that the optimal bound for $k$ should be $k\geq d:=\left \lceil \frac{\mathsf D(G)}{\exp(G)}\right\rceil$. For $p$-groups, this was proven  for $d\leq 4$ when $p\geq 2d-1$ by Dongchun Han \cite{han-skq-rank4}.
For more general $p$-groups, Xiaoyu He could show $\mathsf s_{k\exp(G)}(G)\leq k\exp(G)+\mathsf D(G)-1$ holds for $k\geq p+d$ when $p\geq \frac72 d-\frac32$, and they posed the problem of obtaining a significant improvement of their result by removing the dependance on $p$ from the lower bound for $k$ \cite[pp. 405]{he-skq}.

Our concluding applications are to use Theorem \ref{thm-axkatz-gen} to give a much shorter  proof of Dongchun Han's \cite{han-skq-rank4}  result (Theorem \ref{thm-sk-spec}), and to also answer the problem of Xiaoyu He \cite{he-skq} in the affirmative by showing $k>\frac{d(d-1)}{2}$, which is independent of $p$,  suffices when $p>d(d-1)$ (Theorem \ref{thm-sk-gen}). Both these results make use of Theorem \ref{thm-alg-expconsequence}, which is derived from Theorem \ref{thm-axkatz-gen} and generalizes \cite[Theorem 3]{he-skq} by relaxing the hypothesis $X\subseteq [1,p]$ to that given in \eqref{det-hyp}. Xiaoyu He proved \cite[Theorem 3]{he-skq} by an extension of the method used by Kubertin \cite{kubertin-skq}, which was based on the  methods developed by R\'onyai for his result regarding the Kemnitz Conjecture \cite{ronyai-kemnitz}. In this way,  Theorem \ref{thm-axkatz-gen} simultaneously generalizes both the Chevalley-Warning Theorem and the main applications of the  algebraic method of R\'onyai into a single algebraic tool.

\begin{theorem}\label{thm-alg-expconsequence}
Let $G$ be a finite abelian $p$-group with exponent $q>1$, let $d=\left \lceil \frac{\mathsf D^*(G)}{q}\right\rceil$, let $m\geq 0$,  let $X\subseteq \mathbb N$ be a subset of positive integers with $|X|\geq d+m$, and let $\{x_1,\ldots,x_s\}= [1,\max X]\setminus X$ with the $x_i$ distinct. Suppose \be\label{det-hyp}\prod_{i=1}^{s}x_i\prod_{1\leq i<j\leq s}(x_j-x_i)\not\equiv 0\mod p^{m+1}.\ee Then
\begin{align*}\mathsf s_{ X\cdot q}(G)&\leq \big(\max X-|X|+\frac{m(p-1)}{p}+1\big)q+\mathsf D^*(G)-1\\
&\leq \big(\max X+1-\frac{m}{p}\big)q-r,\end{align*} where $r\in [1,q]$ is the integer such that $d=\frac{\mathsf D^*(G)+r-1}{q}$.
\end{theorem}

\begin{theorem}
\label{thm-sk-spec}
Let $G$ be a finite abelian $p$-group with exponent $q$,  let $d=\left \lceil \frac{\mathsf D^*(G)}{q}\right\rceil$, and suppose $p\geq 2d-1$ and $d\leq 4$. Then  $$\mathsf s_{kq}(G)\leq kq+\mathsf D^*(G)-1\quad \mbox{ for every $k\geq d$}.$$
\end{theorem}

\begin{theorem}
\label{thm-sk-gen}
Let $G$ be a finite abelian $p$-group with exponent $q$,  let $d=\left \lceil \frac{\mathsf D^*(G)}{q}\right\rceil$, and suppose $p>d(d-1)$. Then $$\mathsf s_{kq}(G)\leq kq+\mathsf D^*(G)-1 \quad\mbox{ for every $k>\frac{d(d-1)}{2}$}.$$
\end{theorem}

\subsection{Additional Notation} For our applications in Combinatorial Number Theory, we will have need to deal with (combinatorial) sequences $S$ of terms from a finite abelian group $G$. Here, per tradition in Combinatorial Number Theory, a \emph{sequence} is considered to be a finite and unordered string of elements from $G$, which we write as  $$S=g_1\bdot\ldots\bdot g_\ell$$ with the $g_i\in G$ the terms in the sequence $S$ and each term separated by the concatenation operation $\bdot$. From a combinatorial perspective, a sequence is simply a multi-set, where we use the natural language of sequences to describe its properties, and use the formal algebraic notation from free abelian monoids to easily describe and manipulate its terms \cite{alfred-book} \cite{G-book}. The former avoids confusion with ordinary sets, and the latter is very helpful in more complicated combinatorial arguments.
Then $|S|=\ell$  denotes the length of the sequence $S$. Analogous to the definition of the $p$-adic valuation, for $g\in G$,  $\vp_g(S)$ denotes the multiplicity of the term $g$ in $S$, in which case $S=\prod^\bullet_{g\in G}g^{[\vp_g(S)]}$, where $g^{[n]}={\underbrace{g\bdot\ldots\bdot g}}_n$ denotes the sequence consisting of the element $g$ repeated $n$ times. The notation  $T\mid S$ indicates that  $T$ is a subsequence of $S$, meaning $\vp_g(T)\leq \vp_g(S)$ for all $g\in G$, and then $T^{[-1]}\bdot S$ or $S\bdot T^{[-1]}$ denotes the sequence obtained from $S$ by removing the terms in $T$, so $\vp_g(T^{[-1]}\bdot S)=\vp_g(S)-\vp_g(T)$ for all $g\in G$.
The sum of terms in $S$ is denoted $$\sigma(S)=g_1+\ldots+g_\ell\in G,$$ and the sequence $S$ is \emph{zero-sum} if $\sigma(S)=0$. Given a subset $X\subseteq \N_0$, we use the notation
$$\Sigma_{X}(S)=\{\sigma(T):\; T\mid S,\; |T|\in X\}$$ to denote all elements $g\in G$ that can be represented of a sum of terms from a subsequence of $G$ whose length lies in $X$. In the case $X=\{1,2,\ldots,\}$, we use the abbreviation $$\Sigma(S)=\Sigma_{\{1,2,\ldots\}}(S)=\{\sigma(T):\; T\mid S,\; |T|\geq 1\}$$ to denote all elements that are a sum of terms from a nontrivial subsequence of $S$. The sequence $S$ is called \emph{zero-sum free} if it has no nontrivial zero-sum subsequences, i.e., if $0\notin \Sigma(S)$. For $j\geq 0$, we let $$N_j(S)=|\{I\subseteq [1,\ell]:\; |I|=j,\;\sigma\big({\prod}^\bullet_{i\in I}g_i\big)=0\}|$$ count the number of (indexed) zero-sum subsequences of $S=g_1\bdot\ldots\bdot g_\ell$ with length $j$.

Regarding finite abelian groups $G$, we let $C_n$ denote a cyclic group of order $n\geq 1$. Then $G=C_{n_1}\oplus\ldots\oplus C_{n_r}$ with $1\leq n_1\mid \ldots\mid n_r$ and $n_r=\exp(G)$ the \emph{exponent} of $G$, and we set $$\mathsf D^*(G)=1+\Sum{i=1}{r}(n_i-1).$$ The order of an element $g\in G$ is denoted $\ord(g)$. A \emph{basis} for $G$ is a tuple $(e_1,\ldots,e_r)$ of elements $e_1,\ldots,e_r\in G$ with $G=\la e_1\ra\oplus\ldots\oplus \la e_r\ra$.  Finally, given a subset $X=\{x_1,\ldots,x_s\}\subseteq \Z$ and $q\in \Z$, we let $$X\cdot q=\{x_1q,\ldots,x_sq\}.$$

\section{Proof of the Weighted Ax-Katz-Wilson Theorem}\label{sec-alg}

In this section, we give the details of the proof of Theorem \ref{thm-axkatz-gen}. The following congruence is the first the key component in the proof. The case when $w(X)=1$ is  a constant polynomial  is a result of Weisman \cite{weisman}, generalizing an older congruence of Fleck \cite{Fleck} \cite{Dickson-fleckref} who treated the case $s=1$. The more general version involving the polynomial weight $w(X)$ was originally proved by Daqing Wan \cite{Wan-weighted-weisman-fleck}, with an elementary proof via complex roots of unitary later found by Zhi-Wei Sun and Daqing Wan \cite{Sun-Wan-elem-weight-weis-fleck}.

\begin{theorem}[Weighted Weisman-Fleck Congruence] \label{thm-weisman-Fleck-weighted} Let $n,\,r,\,s\geq 0$  be integers, let $p\geq 2$ be prime, and let $w(X)\in \Q[X]$ be an integer valued polynomial of degree $t\geq 0$. Then
 $$\underset{i\geq 0}{\Summ{i\equiv r\mod p^s}}(-1)^i\binom{n}{i}w\Big(\frac{i-r}{p^s}\Big)\equiv 0\mod p^m,\quad
\mbox{ where
$m=\max\{0,\;\left\lceil \frac{n-(t+1)p^{s}+1}{\varphi(p^s)}\right\rceil$}\}.
$$
\end{theorem}

The set  $$\mathsf{Map}(\Z)=\{f:\Z\rightarrow\Z\}$$ of all maps $f:\Z\rightarrow \Z$ forms an abelian group with addition defined pointwise: $(f+g)(x)=f(x)+g(x)$ for $f,\,g\in \mathsf{Map}(\Z)$ and $x\in\Z$. We then have an endomorphism ring for this abelian group, $$\mathsf{End}(\mathsf{Map}(\Z))=\{F:\mathsf{Map}(\Z)\rightarrow \mathsf{Map}(\Z):\;\mbox{$F$ is an abelian group homomorphism}\},$$
with addition in $\mathsf{End}(\mathsf{Map}(\Z))$ again defined pointwise and multiplication given by composition, so $(FG)(f)=F(G(f))$ and $(F+G)(f)=F(f)+G(f)$ for $F,\,G\in \mathsf{End}(\mathsf{Map}(\Z))$ and $f\in \mathsf{Map}(\Z)$.

Let $I\in \mathsf{End}(\mathsf{Map}(\Z))$ denote the identity map and let $E\in \mathsf{End}(\mathsf{Map}(\Z))$ be the shift operator, defined by $$E(f)(x):=f(x+1)\quad\mbox{ for $f\in \mathsf{Map}(\Z)$ and $x\in \Z$}.$$ The finite difference operator is then the map $$\Delta:=E-I\in \mathsf{End}(\mathsf{Map}(\Z)),$$ meaning $$ \Delta f(x):=\Delta(f)(x)=f(x+1)-f(x)\quad\mbox{ for $f\in \mathsf{Map}(\Z)$ and $x\in \Z$}.$$

The next component in Wilson's argument is the classical Newton Expansion of an integer valued function, which is easily derived from the above set-up. We include the brief proof for the reader's benefit.

\begin{proposition}[Newton Expansion] \label{prop-newton-expan}
For any map  $f:\Z\rightarrow \Z$, we have
\be\label{sum-newt}f(x)=\Sum{n=0}{\infty}(\Delta^nf)(0)\binom{x}{n}\quad\mbox{ for all $x\in\mathbb N_0$}.\ee
\end{proposition}

\begin{proof}
 Iterating the identity $(\Delta+I)f(y)=f(y+1)$, for  $y\in \Z$, it follows that $(\Delta+I)^xf(y)=f(y+x)$ for $y\in\Z$ and $x\geq 0$, whence
\begin{align*}
\Sum{n=0}{\infty}(\Delta^nf)(0)\binom{x}{n}=
\left(\Sum{n=0}{x}\binom{x}{n}\Delta^n\right)f(0)=(\Delta+I)^xf(0)=f(x)
\end{align*} for all $x\in\mathbb N_0$.
\end{proof}

To deal with  general weight functions $w(X)$, we recall the  well-known fact that the integer valued polynomials   $\mathsf{Int}(\Z)\subseteq \Q[X]$ are a free abelian group with basis the binomial functions \cite{int-poly-survey}. This essentially means there is little loss of generality to only consider $w(X)=\binom{X}{t}$,  where $t\geq 0$,  when using  a weight function, or even simply $w(X)=X^t$ for $t\geq 0$ if linear independence is all that is required.

\begin{proposition} \label{prop-int-valued-binom-basis} $\mathsf{Int}(\Z)$ is a free abelian group with basis $\{\binom{X}{t}:\; t=0,1,\ldots\}$.
\end{proposition}

Next, we come to the main step in Wilson's proof, which he modestly named a lemma. The case where $w(X)=1$ is the constant polynomial equal to $1$ is found in Wilson's original paper \cite{Wilson-ax-katz}. Exchanging the use of the non-weighted Weisman-Fleck congruence with its weighted version (Theorem \ref{thm-weisman-Fleck-weighted}) in Wilson's argument, one  obtains the following weighted version with no other major modifications needed. In order to obtain a more self-contained work, we include the details below, which may also be found in an unpublished paper of Zhi-Wei Sun \cite{Sun-Ax-Katz-Wilson}, who was the first to realize Wilson's ideas could readily be extended to include weights.

\begin{theorem}[Weighted Wilson's Lemma]\label{thm-wilson-approx}
Let $m\geq 1$ and $s\geq 0$ be integers, let $p\geq 2$ be prime, let $w(X)\in \Q[X]$ be an integer valued polynomial of degree $t\geq 0$, and let $f:\Z\rightarrow \Z$ be a map that is periodic with period $p^{s}$. Then there exists a rational polynomial $g(X)=\Sum{n=0}{d}a_n\binom{X}{n}\in\Q[X]$ with $a_n\in \Z$ and $d< (t+1)p^s+(m-1)\varphi(p^s)$ such that  \begin{align*}&
g(x)\equiv w\Big(\Big\lfloor \frac{x}{p^s}\Big\rfloor\Big)f(x)\mod p^m\quad\mbox{ for all $x\in \Z$},\quad\und\\ &a_n\equiv 0\mod p^\ell \quad \mbox{ for all $n\in [0,d]$}, \quad\mbox{where $\ell=\max\{0,\;\left\lceil\frac{n-(t+1)p^s+1}{\varphi(p^s)} \right\rceil$}\}.
\end{align*}
\end{theorem}

\begin{proof} Define the map $h:\Z\rightarrow \Z$ by
 $$h(x)=w\Big(\Big\lfloor \frac{x}{p^s}\Big\rfloor\Big)f(x)\quad\mbox{ for $x\in \Z$}$$ and use Proposition \ref{prop-newton-expan} to write
\be\label{sun-day}w\Big(\Big\lfloor \frac{x}{p^s}\Big\rfloor\Big)f(x)=h(x)=\Sum{n=0}{\infty}(\Delta^nh)(0)\binom{x}{n}
\quad \mbox{ for all $x\in \mathbb N_0$}.\ee
Let $I,\,E,\Delta=E-I\in \mathsf{End}(\mathsf{Map}(\Z))$ be as defined earlier. Since $f$ is periodic with period $p^s$, we have $f(i)\equiv f(r)\mod p^s$ whenever $i\equiv r\mod p^s$.
 For any $n\geq 0$, it follows that \begin{align}\nn
(\Delta^nh)(0)&=((E-I)^nh)(0)=\left(\Big(\Sum{i=0}{n}\binom{n}{i}(-I)^{n-i}E^i
\Big)h\right)(0)=
\Sum{i=0}{n}(-1)^{n-i}\binom{n}{i}(E^i h)(0)\\\nn &=
\Sum{i=0}{n}(-1)^{n-i}\binom{n}{i}h(i)=\Sum{r=0}{p^s-1}\underset{i\geq 0}{\Summ{i\equiv r\mod p^s}}(-1)^{n-i}\binom{n}{i}w\Big(\Big\lfloor \frac{i}{p^s}\Big\rfloor \Big)f(i)\\\nn
\label{sun-down}&=\Sum{r=0}{p^s-1}f(r)\left(\underset{i\geq 0}{\Summ{i\equiv r\mod p^s}}(-1)^{n-i}\binom{n}{i}w\Big( \frac{i-r}{p^s}\Big)\right).
\end{align}
Applying Theorem \ref{thm-weisman-Fleck-weighted}, it follows that  \be\nn a_n:=(\Delta^nh)(0)\equiv 0\mod p^\ell, \quad\mbox{where $\ell=\max\{0,\;\left\lceil \frac{n-(t+1)p^{s}+1}{\varphi(p^s)}\right\rceil$\} }.\ee As a particular consequence, we have $a_n\equiv 0\mod p^m$ for all $n\geq (t+1)p^s+(m-1)\varphi(p^s)$. Combined with  \eqref{sun-day}, we obtain
\be\label{sun-dusk}g(x)\equiv h(x)=w\Big(\Big\lfloor\frac{x}{p^s}\Big\rfloor\Big)f(x)\mod p^m\quad\mbox{ for all $x\in\N_0$},\ee where
 $$g(X):=\Sum{n=0}{d}a_n\binom{X}{n}
\in\Q[X]\quad\und\quad d= (t+1)p^s+(m-1)\varphi(p^s)-1.$$
To complete the proof, we need to show  \eqref{sun-dusk} also holds for $x<0$.

 For $n\geq 0$ and $x,\,y\in \Z$, we have  $\binom{x+y}{n}=\binom{x}{n}+y\frac{z}{n!}$,
for some $z\in \Z$, whence\be\label{sunmore}\binom{x+y}{n}\equiv \binom{x}{n}\mod p^m\quad\mbox{ for any $x,\,y\in \Z$ with $\vp_p(y)\geq m+\vp_p(n!)$}.\ee Proposition \ref{prop-int-valued-binom-basis} implies  that $w(X)=\Sum{n=0}{t}b_n\binom{X}{n}$ for some $b_n\in\Z$. Combined with \eqref{sunmore}, we conclude that   \be\label{sunevenmore}w(x+y)\equiv w(x)\mod p^m\quad\mbox{ for any $x,\,y\in \Z$ with $\vp_p(y)\geq m+\vp_p(t!)$}.\ee
Let $x\in \Z$ be arbitrary and let $y\geq 0$ be an integer with $x+y\geq 0$ and  $$\vp_p(y)\geq \max\{s+m+\vp_p(t!),\,m+\vp_p(d!)\}.$$
  Then \begin{align*}g(x)&=\Sum{n=0}{d}a_n\binom{x}{n}\equiv \Sum{n=0}{d}a_n\binom{x+y}{n}=g(x+y)\equiv w\Big(\Big\lfloor \frac{x+y}{p^s}\Big\rfloor\Big)f(x+y)\\
&=w\Big(\Big\lfloor\frac{x}{p^s}\Big\rfloor+\frac{y}{p^s}\Big)f(x)\equiv w\Big(\Big\lfloor\frac{x}{p^s}\Big\rfloor\Big)f(x)\mod p^m,\end{align*}
 which establishes \eqref{sun-dusk}  for $x<0$, completing the proof.
\end{proof}

The following simple lemma is well-known (combine Fermat's Little Theorem \cite{niven-zuck-mont-nr-text} with \cite[Lemma 22.3]{G-book}).

\begin{lemma}
\label{lem-sumofpowers} Let $p\geq 2$ be  prime and let $m\geq 0$ be an integer. Then
$$\Summ{x\in \mathbb F_p}x^m=\left\{
                            \begin{array}{ll}
                              0 & \hbox{if $m\not\equiv 0\mod p-1$} \\
                              -1 & \hbox{if $m\equiv 0\mod p-1$.}
                            \end{array}
                          \right.$$
\end{lemma}

The next lemma is  a variation on Chevalley's key observation used in the proof of the Chevalley-Warning Theorem \cite{chevalley} \cite{warning} \cite{G-book} \cite{nat-book} \cite{Tao-Vu-book}. The case when all $\mathcal I_j=[0,p-1]$ is found in Wilson's original paper \cite{Wilson-ax-katz}, but the argument is sufficiently robust to also work when  replacing  $[0,p-1]$ with an arbitrary complete system of residues modulo $p$. As the added flexibility of being able to consider arbitrary complete system of residues is rather crucial, we include the details.

\begin{lemma}
\label{lemm-wilson-ax-katz-summation-gen}
Let $p\geq 2$ be prime, let $n\geq 1$, let $\mathcal B=\mathcal I_1\times \ldots \times \mathcal I_n$ with each $\mathcal I_j\subseteq \Z$ for $j\in [1,n]$ a complete system of residues  modulo  $p$, and suppose $f\in\Q[X_1,\ldots,X_n]$ is an integer valued polynomial with $\deg_j(f)\leq p-2$ for every $j\in[1,n]$, and  $\vp_p(c)\geq 0$ for every coefficient $c\in \Q$ of a monomial in $f(\textbf x)$.  Then  $$\Summ{\textbf a\in \mathcal  B}f(\textbf a)\equiv 0\mod p^n.$$
\end{lemma}

\begin{proof}
Let $g(\textbf x)=cX_1^{k_1}X_2^{k_2}\cdots X_n^{k_n}$ be an arbitrary monomial occurring in $f(\textbf x)$, so   $c_g\in \Q\setminus\{0\}$ and $\vp_p(c_g)\geq 0$ by hypothesis. Now
\begin{align*}\Summ{\textbf a\in \mathcal B}g(\textbf a)&=\Summ{\textbf (a_1,\ldots,a_n)\in \mathcal B}c_ga_1^{k_1}a_2^{k_2}\cdots a_n^{k_n}=\Summ{(a_1,\ldots,a_{n-1})\in \mathcal B'}\left(c_ga_1^{k_1}\cdots a_{n-1}^{k_{n-1}}\Summ{a_n\in \mathcal I_n}
a_n^{k_n}\right),\end{align*} where $\mathcal B'=\mathcal I_1\times\ldots\times \mathcal I_{n-1}$.
By hypothesis, we have $k_j\leq p-2$ for every $j\in[1,n]$. Combined with the hypothesis that $\mathcal I_n$ is a complete system of residues modulo $p$, we can apply  Lemma \ref{lem-sumofpowers} to conclude that  $\Summ{a_n\in \mathcal I_n}
a_n^{k_n}=b'p$ for some $b'\in \Z$. Consequently, $$\Summ{\textbf a\in \mathcal B}g(\textbf a)=b'p\Summ{\textbf a\in \mathcal B'}h(\textbf a),$$ where $h(\textbf x)=c_gX_1^{k_1}\cdots X_{n-1}^{k_{n-1}}\in\Q[X_1,\ldots,X_{n-1}]$.
Iterating this argument $n$ times, it follows that   $$\Summ{\textbf a\in \mathcal B}g(\textbf a)=c_gb_gp^n\quad\mbox{ for some $b_g\in\Z$}.$$
Thus $\Summ{\textbf a\in \mathcal B}f(\textbf a)=\Summ{g}\Summ{\textbf a\in \mathcal B} g(\textbf a)=\Big(\Summ{g} c_gb_g\Big)p^n$, where the sum $\Summ{g}$ is taken over all monomials $g$ occurring in $f$. Hence, since $f$ is integer valued with $b_g\in \Z$ and $\vp_p(c_g)\geq 0$ for all $g$, it follows that $\Summ{\textbf a\in \mathcal B}f(\textbf a)\equiv 0\mod p^n$, as desired.
\end{proof}

The final component in Wilson's argument is the following consequence of Lemma \ref{lemm-wilson-ax-katz-summation-gen}. Again, the case when all $\mathcal I_j=[0,p-1]$ is found in Wilson's original paper \cite{Wilson-ax-katz}, and the more general case simply requires using   Lemma \ref{lemm-wilson-ax-katz-summation-gen} in Wilson's original argument, with the details given below.

\begin{lemma}
\label{lemm-wilson-ax-katz-summation}
Let $p\geq 2$ be prime, let $n\geq 0$, let $\mathcal B=I_1\times\ldots\times I_n$ with each $I_j\subseteq \Z$ for $j\in[1,n]$ a  complete system of residues  modulo  $p$, let $f_1,\ldots,f_s\in \Z[X_1,\ldots,X_n]$ be  nonzero polynomials,  and suppose   \be\label{monomial-sum} f(\textbf x)=\binom{f_1(\textbf x)}{k_1}\binom{f_2(\textbf x)}{k_2}\cdots \binom{f_s(\textbf x)}{k_s}\in \Q[X_1,\ldots,X_n]\ee for some $k_1,\ldots, k_s\geq 0$ and $s\geq 1$. If $n\geq (m-1)+\frac{\deg f+1}{p-1}$, where $m\geq 1$, then $$\Summ{\textbf a\in \mathcal B}f(\textbf a)\equiv 0\mod p^m.$$
\end{lemma}

\begin{proof}
For $k\geq0$ and $t\geq 1$, we utilize the polynomial identity \be\label{sum-id}\binom{Y_1+\ldots+Y_t}{k}=\underset{(k_1,\ldots,k_t)\in \mathbb N_0^t}{\Summ{k_1+\ldots+k_t=k}}\binom{Y_1}{k_1}\cdots \binom{Y_t}{k_t},\ee
which holds when each $Y_i>0$ is an integer by a basic combinatorial counting argument, and extends to the case when each $Y_i$ is a polynomial by noting that the difference of both sides is then a polynomial with all $\textbf a\in \mathbb N^t$ as roots. We can write each  $f_j(\textbf x)\in\Z[X_1,\ldots,X_n]$, for $j\in [1,s]$, as a sum of $t_j\geq 1$ nonzero monomials with integer coefficients, and then use the identity given in \eqref{sum-id} to write $f(\textbf x)$ as a sum of expressions of the form given in \eqref{monomial-sum} (with $s$ replaced by $\Sum{j=1}{s}t_j$ and the $k_i$ varying), with each such expression in the sum  individually satisfying the hypotheses of the lemma and having each  $f_j(\textbf x)$ occurring in a given expression replaced by a single nonzero monomial. As it would then suffice to prove the lemma individually for each of the expressions in this sum,
it follows that we can w.l.o.g. assume each $f_j(\textbf x)$ is itself a monomial.
As a result,  it follows that there is a unique monomial in $f(\textbf x)$ whose degree equals $\deg f$, namely, the monomial $$h(\textbf x):=\frac{1}{k_1!\cdots k_s!}f_1(\mathbf x)^{k_1}\cdots f_s(\mathbf x)^{k_s}.$$ Additionally, any  monomial $cX_1^{b_1}\cdots X_s^{b_s}$ occurring in $f(\textbf x)$ must have $b_j\leq \deg_j(h)$ for all $j\in [1,s]$.

By hypothesis,  $\deg f\leq (n-m+1)(p-1)-1$, which combined with the  Pigeonhole Principle means there  are at most $n-m$ variables $X_j$ having  $\deg_j(h(\textbf x))\geq p-1$. By re-indexing, we can w.l.o.g. assume that $\deg_j(h(\textbf x))\leq p-2$ for every $j\in [1,m]$. Since every monomial in $f(\textbf x)$ has its degree in the variable $X_j$ bounded by $\deg_j(h(\textbf x))$,  we conclude that \be\label{deg-bounding}\deg_j(f(\textbf x))\leq p-2\quad \mbox{ for all $j\in [1,m]$}.\ee
This has the useful consequence that any variable $X_j$ with $j\in [1,m]$ cannot  occur with positive degree in any monomial $f_i(\textbf x)$ having $k_i\geq p-1$.

We can write \be\label{impliy}\Summ{\textbf a\in \mathcal B}f(\textbf a)=\Summ{\textbf b\in \mathcal I_{m+1}\times\ldots \times \mathcal I_{n}}\Summ{\textbf c\in \mathcal I_1\times \ldots\times \mathcal I_m}
f_{\textbf b}(\textbf c),\ee
where $f_{\textbf b}(\textbf x)=f(X_1,\ldots, X_m,b_{m+1},\ldots, b_{n})\in \Q[X_{1},\ldots, X_m]$ for $\textbf b=(b_{m+1},\ldots, b_{n})$. Then \be\label{prod-def}f_\textbf b(\textbf x)=\binom{f_1(X_1,\ldots,X_m,b_{m+1},\ldots,b_n)}{k_1}\cdots \binom{f_s(X_1,\ldots,X_m,b_{m+1},\ldots,b_n)}{k_s}\ee is a polynomial in the variables $X_1,\ldots, X_m$. Moreover, in view of \eqref{deg-bounding}, we have $$\deg_j f_\textbf b\leq p-2\quad\mbox{ for all $j\in [1,m]$}.$$ From \eqref{prod-def} and the fact that $f_i\in\Z[X_1,\ldots,X_n]$ for all $i\in[1,s]$, we see that $f_\textbf b\in \Q[X_1,\ldots,X_m]$ is an integer valued polynomial.

Let $\textbf b=(b_{m+1},\ldots, b_{n})\in \mathcal I_{m+1}\times \ldots\times \mathcal I_n$ be arbitrary. In view of \eqref{prod-def}, $f_\textbf b(\textbf x)$ is a product of $s$ factors of the form
$\binom{f_i(X_1,\ldots,X_m,b_{m+1},\ldots,b_n)}{k_i}$, for $i\in [1,s]$. If $k_i\geq p$,  then none of the variables $X_1,\ldots, X_m$ occur with positive degree in $f_i(\textbf x)$, as already noted, meaning the factor $\binom{f_i(X_1,\ldots,X_m,b_{m+1},\ldots,b_n)}{k_i}$ is a constant, which must then be an integer since $\binom{f_i(\textbf x)}{k_i}$ is an integer valued polynomial (in view of $f_i\in\Z[X_1,\ldots,X_n]$).
From this, and the fact that all $f_i\in\Z[X_1,\ldots,X_n]$,  we conclude that the every coefficient $c$ of a monomial in $f_\textbf b(\textbf x)$ must have the denominator of its coefficient $c$ dividing $\prod_{i\in J}k_i!$, where $J\subseteq [1,s]$ is the subset of all indices $i\in[1,s]$ with $k_i\leq p-1$, which ensures that  $\vp_p(c)\geq 0$ (as $p$ is prime). Combined with the conclusions of the previous paragraph, we can  now apply Lemma \ref{lemm-wilson-ax-katz-summation-gen} to $f_\textbf b$ to conclude that $$\Summ{\textbf c\in \mathcal I_1\times\ldots\times \mathcal I_m}
f_{\textbf b}(\textbf c)\equiv 0\mod p^m\quad\mbox{for all $\textbf b\in \mathcal I_{m+1}\times\ldots\times \mathcal I_n$},$$ which combined with \eqref{impliy} yields the desired congruence.
\end{proof}

We can now complete the proof of Theorem \ref{thm-axkatz-gen}.

\begin{proof}[Proof of Theorem \ref{thm-axkatz-gen}] The hypotheses give  \begin{align} n>(m-1)\max_{i\in [1,s]}\left\{1,\;\frac{\varphi(p^{m_i})}{p-1}\deg f_i\right\}+ \Sum{i=1}{s}\frac{(t_i+1)p^{m_i}-1}{p-1}\deg f_i.\label{lagoon4}\end{align}

For each $j\in[1,s]$, apply  Theorem \ref{thm-wilson-approx} to the integer valued function with period $p^{m_j}$ which sends $0$ to $1$ and all elements of $[1,p^{m_j}-1]$ to $0$, using $w_j(X)$ as weight function, to find a rational polynomial $$g_j(X)=\Sum{i=0}{d_j} b^{(j)}_i\binom{X}{i}\in\Q[X],$$  with all $b_i^{(j)}\in\Z$ and $d_j\leq
(t_j+1)p^{m_j}+(m-1)\varphi(p^{m_j})-1$,  such that
\begin{align}&g_j(x)\equiv \left\{                                              \begin{array}{ll}
    w_j\big(\frac{x}{p^{m_j}}
\big)\mod p^m & \hbox{if $x\equiv 0\mod p^{m_j}$} \\
                                                   0 \mod p^m & \hbox{if $x\not\equiv 0\mod p^{m_j}$,}
                                                 \end{array}
                                               \right. \quad\und\\
& b^{(j)}_i\equiv 0\mod p^\ell, \quad\mbox{where \label{song} $\ell=\max\{0,\;\left\lceil\frac{i-(t_j+1)p^{m_j}+1}{\varphi(p^{m_j})} \right\rceil$}\}.\end{align}
In view of all definitions involved, \begin{align}\nn N&\equiv \Summ{\textbf a\in \mathcal B}g_1\big(f_1(\textbf a)\big)g_2\big(f_2(\textbf a)\big)\cdots g_s\big(f_s(\textbf a)\big)\mod p^m\\\nn
&=\Summ{\textbf a\in\mathcal B}\left(\Sum{i=0}{d_1}b^{(1)}_i\binom{f_1(\textbf a)}{i}\right)\left(\Sum{i=0}{d_2}b^{(2)}_i\binom{f_2(\textbf a)}{i}\right)\cdots \left(\Sum{i=0}{d_s}b^{(s)}_i\binom{f_s(\textbf a)}{i}\right)\\
&=\Summ{(k_1,\ldots,k_s)\in \prod_{i=1}^s[0,d_i]}b^{(1)}_{k_1}b^{(2)}_{k_2}\cdots b^{(s)}_{k_s}\Summ{\textbf a\in \mathcal B}\binom{f_1(\textbf a)}{k_1}\binom{f_2(\textbf a)}{k_2}\cdots \binom{f_s(\textbf a)}{k_s}.\label{V-cong}
\end{align}

It suffices to show each summand in \eqref{V-cong} is divisible by $p^m$. With this goal in mind,
let $(k_1,\ldots,k_s)\in \prod_{i=1}^s[0,d_i]$ be arbitrary.  For $j\in [1,s]$, define $\ell_j:=\max\{0,\,\left\lceil\frac{k_j-(t_j+1)p^{m_j}+1}{\varphi(p^{m_j})}
 \right\rceil\}\geq \frac{k_j-(t_j+1)p^{m_j}+1}{\varphi(p^{m_j})}$, in which case  \be\label{lagoon1}k_j\leq \ell_j\varphi(p^{m_j})+(t_j+1)p^{m_j}-1.\ee All summands in \eqref{V-cong} with $\ell_1+\ldots+\ell_s\geq m$ are congruent to $0$ modulo $p^m$ by \eqref{song}, since this ensures that $b_{k_1}^{(1)}\cdots b_{k_s}^{(s)}\equiv 0\mod p^m$. We need only consider those with  \be\label{lagoon3}\ell_1+\ldots+\ell_s=m-t\quad\mbox{ for some $t\geq 1$}.\ee In this case, \eqref{song} instead ensures that the coefficient $b^{(1)}_{k_1}b^{(2)}_{k_2}\cdots b^{(s)}_{k_s}$ is divisible by $p^{m-t}$, so we just need to show that the summation $\Summ{\textbf a\in \mathcal B}\binom{f_1(\textbf a)}{k_1}\binom{f_2(\textbf a)}{k_2}\cdots \binom{f_s(\textbf a)}{k_s}$ is divisible by $p^{t}$.

In view of \eqref{lagoon1}, \eqref{lagoon3}, \eqref{lagoon4} and $t\geq 1$, we have
\begin{align*}
&\deg\left(\binom{f_1(\textbf x)}{k_1}\binom{f_2(\textbf x)}{k_2}\cdots \binom{f_s(\textbf x)}{k_s}\right)=k_1\deg f_1+\ldots+k_s\deg f_s\\
&\leq \Sum{j=1}{s}\Big(\ell_j\varphi(p^{m_j})+(t_j+1)p^{m_j}-1\Big)\deg f_j
\\
&=(p-1)\Big(\Sum{i=1}{s}\ell_i\frac{\varphi(p^{m_i})}{p-1}\deg f_i+\Sum{i=1}{s}\frac{(t_i+1)p^{m_i}-1}{p-1}\deg f_i
\Big)\\
&\leq (p-1)\Big((\ell_1+\ldots+\ell_s)\max_{i\in [1,s]}\left\{1,\;\frac{\varphi(p^{m_i})}{p-1}\deg f_i\right\}+\Sum{i=1}{s}\frac{(t_i+1)p^{m_i}-1}{p-1}\deg f_i
\Big)\\
&=(p-1)\Big((m-1-(t-1))\max_{i\in [1,s]}\left\{1,\;\frac{\varphi(p^{m_i})}{p-1}\deg f_i\right\}+\Sum{i=1}{s}\frac{(t_i+1)p^{m_i}-1}{p-1}\deg f_i
\Big)
\\
&\nn < (p-1)\Big(n-(t-1)\max_{i\in [1,s]}\left\{1,\;\frac{\varphi(p^{m_i})}{p-1}\deg f_i\right\}\Big)\leq (p-1)(n+1-t),
\end{align*}
implying that $$n\geq (t-1)+\frac{\deg\left(\binom{f_1(\textbf x)}{k_1}\binom{f_2(\textbf x)}{k_2}\cdots \binom{f_s(\textbf x)}{k_s}\right)+1}{p-1}.$$ But now  Lemma \ref{lemm-wilson-ax-katz-summation} implies  that $\Summ{\textbf a\in \mathcal B}\binom{f_1(\textbf x)}{k_1}\binom{f_2(\textbf x)}{k_2}\cdots \binom{f_s(\textbf x)}{k_s}$ is divisible by $p^{t}$, completing the proof as already noted.
\end{proof}

To effectively use Theorem \ref{thm-axkatz-gen} requires a ``good'' choice for the  complete system of residues modulo $p$. This can generally be achieved by use of Hensel's Lemma \cite{niven-zuck-mont-nr-text}. We state one commonly used version below.

\begin{theorem}[Hensel's Lemma] \label{thm-hensels-lem}Let $p\geq 2$ be  prime, let $m\geq 1$ and $e\in [1,m]$ be integers,  and let $f(X)\in \Z[X]$ be a polynomial. If $f(x)\equiv 0\mod p^m$ and $f'(x)\not\equiv 0\mod p$, where $x\in \Z$, then there is some $y\in \Z$ with $$y\equiv x\mod p^m\quad\und\quad f(y)\equiv 0\mod p^{m+e}.$$ Moreover, the value of $y$ is uniquely determined modulo $p^{m+e}$.
\end{theorem}

We conclude the section by giving the short derivation of Proposition \ref{prop-Ax-Katz-ChooseI} using Hensel's Lemma, which provides the appropriate choice for the complete system of residues for many combinatorial applications of Theorem \ref{thm-axkatz-gen}.

\begin{proof}[Proof of Proposition \ref{prop-Ax-Katz-ChooseI}]
  Let $z\in [1,p-1]$ be a primitive residue class modulo the prime $p$, meaning $\{0\}\cup \{z^i:\;i\in[1,p-1]\}$  is a complete system of residues modulo $p$ (since $\Z/p\Z$ is a finite field with cyclic multiplicative group, such $z$ exists) and $$z^{p-1}\equiv 1\mod p.$$   Let $$f(X)=X^{p-1}-1\in \Z[X]$$ and note that  $f'(x)=(p-1)x\equiv -x\not\equiv 0\mod p$ for any $x\in \Z$ with $x\not\equiv 0\mod p$.
For each $i\in [1,p-1]$, we have  $f(z^i)=(z^{p-1})^i-1\equiv 1^i-1=0\mod p$. Thus we can repeatedly apply  Hensel's Lemma (Theorem \ref{thm-hensels-lem}) to find  some $y_i\in [0,p^m-1]$ with \be\label{needprop}y_i\equiv z^i\not\equiv 0\mod p\quad\und\quad y_i^{p-1}-1=f(y_i)\equiv 0\mod p^m,\ee for all $i\in [1,p-1]$. Let $\mathcal I=\{0\}\cup \{y_i:\; i\in [1,p-1]\}$. Since $\{0\}\cup \{z^i:\;i\in[1,p-1]\}$  was a complete system of residues modulo $p$ with $y_i\equiv z^i\mod p$ for all $i$,  it follows that $\mathcal I$ remains a complete system of residues modulo $p$, and one with the needed properties in view of \eqref{needprop}.
\end{proof}

\section{Applications in Combinatorial Number Theory}

In this section, we give the proofs of the applications of Theorem \ref{thm-axkatz-gen}.

\begin{proposition}\label{prop-altsum}
Let $G$ be a finite abelian $p$-group with exponent $q>1$,  and let $S$ be a sequence of terms from $G$ with $|S|\geq m\frac{(p-1)q}{p}+\mathsf D^*(G)$, where $m\geq 0$. Then $$\Sum{j=0}{\infty}(p-1)^{j}N_{j}(S)\equiv 0\mod p^{m+1}.$$
\end{proposition}

\begin{proof}
Write $G=C_{q_1}\oplus \ldots\oplus C_{q_r}$ with each $q_i$ a power of $p$ and $$1<q_1\leq \ldots\leq q_r=q.$$ Then $\mathsf D^*(G)=\Sum{i=1}{r}(q_i-1)+1$. Let $(e_1,\ldots,e_r)$ be a basis for $G$ with $\ord(e_i)=q_i$ for $i\in[1,r]$.
Let $S=g_1\bdot \ldots\bdot g_\ell$, so  $\ell=|S|\geq m\frac{(p-1)q}{p}+\mathsf D^*(G)$.
 For each $i\in [1,\ell]$, write $$g_i=\Sum{j=1}{r}a_i^{(j)}e_j\quad\mbox{ with $a_i^{(j)}\in [0,q_j-1]$}. $$
Let $$f_j(\textbf x)=\Sum{i=1}{\ell}a_i^{(j)}X_i^{p-1}\in \Z[X_1,\ldots,X_\ell], \quad\mbox{ for $j\in[1,r]$}.$$
In view of Proposition \ref{prop-Ax-Katz-ChooseI}, let $\mathcal I\subseteq [0,q-1]$ be a complete system of residues modulo $p$ such that \be\label{fermat-inproof-aa}x^{p-1}\equiv
\left\{
  \begin{array}{ll}
    1 \mod q& \hbox{if $x\not\equiv 0\mod p$} \\
    0\mod q & \hbox{if $x\equiv 0\mod p$,}
  \end{array}
\right.\quad\mbox{ for every $x\in \mathcal I$}.\ee
Observe that  $\max_{j\in [1,r]}\big\{1,\;\frac{\varphi(q_j)}{p-1}\deg f_j\big\}=\max_{j\in [1,r]}\big\{\varphi(q_j)\big\}=\varphi(q)=\frac{(p-1)q}{p}$ and
\begin{align*}\ell=|S|\geq m\max_{j\in [1,r]}\big\{1,\;\frac{\varphi(q_j)}{p-1}\deg f_j\big\}+\Sum{j=1}{r}\frac{q_j-1}{p-1}\deg f_j+1.\end{align*} Thus we can apply Theorem  \ref{thm-axkatz-gen}, with $m$ taken to be $m+1$,  taking $\mathcal I_j=\mathcal I$ for all $j$, and using the polynomials $f_1,\ldots,f_r$, prime powers $q_1,\ldots,q_r=q$, and weight functions  $w_j(X)=1$ for all $j\in [1,r]$.
As a result, letting $$V=\{\textbf a\in \mathcal I^\ell:\; f_j(\textbf a)\equiv 0\mod q_j\;\mbox{ for all $j\in [1,r]$}\},$$
it follows that \be\label{equiv}|V|\equiv 0\mod p^{m+1}.\ee Let  us next describe what $|V|$  equals in terms of the zero-sum subsequences of $S$.

Associate to each $\textbf a\in \mathcal I^\ell$ the subsequence $S_\textbf a=\prod_{j\in I_\textbf a}^\bullet g_j$, where $I_\textbf a\subseteq [1,\ell]$ consists of all $j\in [1,\ell]$ for which the $j$-th coordinate of $\textbf a$ is nonzero modulo $p$.
Thus the nonzero (modulo $p$) terms in $\mathcal I$ ``select'' the terms included in the sequence $S_\textbf a$.
In view of \eqref{fermat-inproof-aa}, the conditions $f_j(\textbf a)\equiv 0\mod q_j$ in the definition of $V$, for $j\in[1,r]$, restrict to tuples $\textbf a\in \mathcal I^\ell$ for which the associated sequence $S_\textbf a$ is zero-sum.  This means that  the tuples $\textbf a\in V$ are precisely those whose associated sequence $S_\textbf a$ is a zero-sum subsequence, in which case  $|S_\textbf a|=j$ for some $j\geq 0$.
Moreover, each zero-sum subsequence of length $j$ is associated to exactly $(p-1)^{j}$ tuples $\textbf a\in \mathcal I^\ell$, for there are $(p-1)$ elements of $\mathcal I$ that are nonzero modulo $p$, each of which selects one term in $S_\textbf a$, while the unique element of $\mathcal I$ congruent to zero is the only way to \emph{not} select a term in $S_\textbf a$. As a result, $|V|=\Sum{j=0}{\infty}(p-1)^{j}N_{j}(S)$,  which combined with \eqref{equiv} yields the desired conclusion.
\end{proof}

We now can complete the proof regarding the Davenport Constant. Note, the original proof also proceeds by first deriving Proposition \ref{prop-altsum}, so the proof below continues as usual.

\begin{proof}[Proof of Theorem \ref{thm-dav-pgroup}]Let $G=C_{q_1}\oplus \ldots\oplus C_{q_s}$ and
let $(e_1,\ldots,e_s)$ be a basis for $G$ with $\ord(e_i)=q_i$ for all $i\in [1,s]$. We can assume $G$ is nontrivial else $\mathsf D(G)=\mathsf D^*(G)=1$. Now $\mathsf D^*(G)=1+\Sum{i=1}{s}(q_i-1)$ and the sequence $\prod^\bullet_{i\in [1,s]}e_i^{[q_i-1]}$ is zero-sum free, showing $\mathsf D(G)\geq \mathsf D^*(G)$. To show the upper bound, let $S$ be a sequence of terms from $G$ with length $\mathsf D^*(G)$. Assuming by contradiction that $S$ is zero-sum free, we obtain $N_i(S)=0$ for all $i>0$, in which case Proposition \ref{prop-altsum} applied with $m=0$ yields the contradiction $1=N_0\equiv 0\mod p$.
\end{proof}

As a minor variation on Proposition \ref{prop-altsum}, we have the following result.

\begin{proposition}\label{prop-altsum-q}
Let $G$ be a finite abelian $p$-group with exponent $q>1$,  let $\alpha\in [0,q-1]$, let $t\geq 1$, and let $S$ be a sequence of terms from $G$ with $|S|\geq m\frac{(p-1)q}{p}+tq-1+\mathsf D^*(G)$, where $m\geq 0$. Then $$\Sum{j=0}{\infty}(p-1)^{jq+\alpha}\Big(j^iN_{jq+\alpha}(S)\Big)\equiv 0\mod p^{m+1},\quad \mbox{ for every $i\in [0,t-1]$}.$$
\end{proposition}

\begin{proof}
Write $G=(\Z/q_1\Z)\oplus \ldots\oplus (\Z/q_r\Z)$ with each $q_i$ a power of $p$ and $$1<q_1\leq \ldots\leq q_r=q_{r+1}:=q.$$ Then $\mathsf D^*(G)=\Sum{i=1}{r}(q_i-1)+1$. Let $(e_1,\ldots,e_r)$ be a basis for $G$ with $\ord(e_i)=q_i$ for $i\in[1,r]$.
Let $S=g_1\bdot \ldots\bdot g_\ell$, so  $\ell=|S|\geq m\frac{(p-1)q}{p}+tq-1+\mathsf D^*(G)$.
 For each $i\in [1,\ell]$, write $$g_i=\Sum{j=1}{r}a_i^{(j)}e_j\quad\mbox{ with $a_i^{(j)}\in [0,q_j-1]$}. $$
Let $$f_j(\textbf x)=\Sum{i=1}{\ell}a_i^{(j)}X_i^{p-1}\in \Z[X_1,\ldots,X_\ell], \quad\mbox{ for $j\in[1,r]$}.$$
Let $$f_{r+1}(\textbf x)=\Sum{i=1}{\ell}X_i^{p-1}-\alpha\in \Z[X_1,\ldots,X_\ell].$$
For $i\in [0,t-1]$, let $$w_i(X)=X^i\in \Z[X].$$
In view of Proposition \ref{prop-Ax-Katz-ChooseI}, let $\mathcal I\subseteq [0,qp^{m+1}-1]$ be a complete system of residues modulo $p$ such that \be\label{fermat-inproof-a}x^{p-1}\equiv
\left\{
  \begin{array}{ll}
    1 \mod qp^{m+1}& \hbox{if $x\not\equiv 0\mod p$} \\
    0\mod qp^{m+1} & \hbox{if $x\equiv 0\mod p$,}
  \end{array}
\right.\quad\mbox{ for every $x\in \mathcal I$}.\ee
Observe that  $\max_{j\in [1,r+1]}\big\{1,\;\frac{\varphi(q_j)}{p-1}\deg f_j\big\}=\max_{j\in [1,r+1]}\big\{\varphi(q_j)\big\}=\varphi(q)=\frac{(p-1)q}{p}$ and
\begin{align*}\ell=|S|&\geq m\max_{j\in [1,r+1]}\big\{1,\;\frac{\varphi(q_j)}{p-1}\deg f_j\big\}+\frac{tq-1}{p-1}\deg f_{r+1}+\Sum{j=1}{r}\frac{q_j-1}{p-1}\deg f_j+1.\end{align*} Thus, for each $i\in [0,t-1]$, we can apply Theorem  \ref{thm-axkatz-gen}, with $m$ taken to be $m+1$,  taking $\mathcal I_j=\mathcal I$ for all $j$, and using the polynomials $f_1,\ldots,f_r,f_{r+1}$, weights $\underbrace{w_0,\ldots,w_0}_r,w_i$, and prime powers $q_1,\ldots,q_r,q_{r+1}=q$.
As a result, letting $$V=\{\textbf a\in \mathcal I^\ell:\; f_j(\textbf a)\equiv 0\mod q_j\;\mbox{ for all $j\in [1,r+1]$}\},$$
it follows that the weighted size of $V$ is congruent to $0$ modulo $p^{m+1}$, for each $i\in [0,t-1]$. Let  us next describe what this size equals.

Associate to each $\textbf a\in \mathcal I^\ell$ the subsequence $S_\textbf a=\prod_{j\in I_\textbf a}^\bullet g_j$, where $I_\textbf a\subseteq [1,\ell]$ consists of all $j\in [1,\ell]$ for which the $j$-th coordinate of $\textbf a$ is nonzero modulo $p$. In view of \eqref{fermat-inproof-a}, the conditions $f_j(\textbf a)\equiv 0\mod q_j$ in the definition of $V$, for $j\in[1,r]$, restrict to tuples $\textbf a\in \mathcal I^\ell$ for which the associated sequence $S_\textbf a$ is zero-sum. Likewise, the additional condition $f_{r+1}(\textbf a)\equiv 0\mod q$ further restricts to tuples $\textbf a\in \mathcal I^\ell$ whose associated sequence $S_\textbf a$ has length $|S_\textbf a|=|I_\textbf a|\equiv \alpha\mod q$. This means that  the tuples $\textbf a\in V$ are precisely those whose associated sequence $S_\textbf a$ is a zero-sum subsequence of length $|S_\textbf a|\equiv \alpha\mod q$, meaning $|S_\textbf a|=jq+\alpha$ for some $j\geq 0$.
Moreover, each zero-sum subsequence of length $jq+\alpha$ is associated to exactly $(p-1)^{jq+\alpha}$ tuples $\textbf a\in \mathcal I^\ell$, and the weighted size of each such tuple is $w_i(j)\equiv j^i\mod p^{m+1}$ (in view of \eqref{fermat-inproof-a}). As a result, for $i\in [0,t-1]$, the weighted size of $V$ equals $\Sum{j=0}{\infty}(p-1)^{jq+\alpha}\Big(j^iN_{jq+\alpha}(S)\Big)$ modulo $p^{m+1}$,  meaning the conclusion of Theorem \ref{thm-axkatz-gen} is precisely the desired conclusion of the proposition.
\end{proof}

We now give the proof of the Kemnitz Conjecture, which contains Alon and Dubiner's  argument that $N_{3p}(S)\neq 0$ implies $N_p(S)\neq 0$ \cite{alon-dubiner-org}. We remark that it would also be possible to derive the congruences below using the higher order $p$ divisibility of $|V|$ in Theorem \ref{thm-axkatz-gen} (combined with combinatorial double counting arguments of the type used by Reiher \cite{reiher-kemnitz-conj}) rather than the weight functions. However, using the weight functions directly is  simpler.

\begin{proof}[Proof of Theorem \ref{thm-kemnitzconj}]
Let $G=C_p^2$ with $(e_1,e_2)$ a basis for $G$. Note that  $0^{[p-1]}\bdot e_1^{[p-1]}\bdot e_2^{[p-1]}\bdot (e_1+e_2)^{[p-1]}$ is a sequence  of $4p-4$ terms from $G$ containing no $p$-term zero-sum subsequence, showing $\mathsf s_p(C_p^2)\geq 4p-3$. To show the upper bound, assume by contradiction that  $S$ is a sequence of terms from $G$ with $|S|=4p-3$ and  $0\notin\Sigma_{p}(S)$. If $p=2$, then $|S|= 4p-3=5$ ensures via the Pigeonhole Principle that $S$ contains a term $g$ with multiplicity at least two, in which case $g^{[2]}$ will be a $p$-term zero-sum subsequence, contrary to assumption. Therefore we can assume $p\geq 3$.

If $T\mid S$ is any subsequence with $|T|\geq 3p-2$, then Proposition  \ref{prop-altsum-q} (applied with $\alpha=0$, $m=0$ and $t=1$) implies that $N_0(T)-N_p(T)+N_{2p}(T)-N_{3p}(T)\equiv 0\mod p$. In particular, since $N_p(S)=0$ by assumption, it follows that any zero-sum subsequence $T\mid S$ with $|T|=3p$ has $N_{2p}(T\bdot g^{[-1]})\equiv -N_0(T\bdot g^{[-1]})=-1\mod p$, for any $g\in\supp(T)$, ensuring that $T\bdot g^{[-1]}$ has a zero-sum subsequence $R$ of length $2p$. However, the complement of $R$ in $T$ would then be a zero-sum subsequence of length $|T|-|R|=p$, contradicting that $N_p(S)=0$. Therefore we instead conclude that \be\label{p2pzero}N_p(S)=N_{3p}(S)=0,\ee and now Proposition  \ref{prop-altsum-q} implies that
\be\label{np-1}N_{2p}(T)\equiv -1\mod p\quad\mbox{ for all $T\mid S$ with $|T|\geq 3p-2$}.\ee For $j\geq 0$, let
$$N_j=N_j(S).$$

Since $|S|\geq 3p-2$, Proposition \ref{prop-altsum-q} (applied with $\alpha=p-1$, $m=0$ and $t=1$) implies that
\be\label{f1} N_{p-1}-N_{2p-1}+N_{3p-1}\equiv 0\mod p\ee

Let $T\mid S$ be an arbitrary zero-sum sequence with $|T|=3p-1$. Then $N_{p-1}(T)=N_{2p}(T)\equiv -1 \mod p$ by \eqref{np-1}, with the first equality holding since the complement in $T$ of a zero-sum subsequence of $T$ is also zero-sum. Thus $\Summ{T}N_{p-1}(T)\equiv -N_{3p-1}\mod p$, where the sum is taken over all zero-sum subsequences $T\mid S$ with $|T|=3p-1$. On the other hand, every zero-sum subsequence $R\mid S$ with $|R|=p-1$ is contained in exactly $N_{2p}(S\bdot R^{[-1]})$ zeros-sum subsequences $T\mid S$ with $|T|=3p-1$. Since $|S\bdot R^{[-1]}|=3p-2$, \eqref{np-1} ensures that $N_{2p}(S\bdot R^{[-1]})\equiv -1\mod p$ for any such $R$, in which case $-N_{3p-1}\equiv \Summ{T}N_{p-1}(T)=\Summ{R}N_{2p}(S\bdot R^{[-1]})\equiv -N_{p-1}\mod p$, where the second sum is taken over all zero-sum subsequences $R\mid S$ with $|R|=p-1$.  Hence
\be\label{f2} N_{p-1}\equiv N_{3p-1}\mod p.\ee

Observe that $N_j(S\bdot 0)=N_j+N_{j-1}$ for every $j>0$. Thus, since $|S\bdot 0|=|S|+1= 4p-2$,  applying Proposition \ref{prop-altsum-q} (with $\alpha=0$, $m=0$ and $t=2$) to $S\bdot 0$ implies
$$N_p+N_{p-1}-2N_{2p}-2N_{2p-1}+3N_{3p-1}+3N_{3p}\equiv 0\mod p.$$ We have $N_{2p}\equiv -1\mod p$ by \eqref{np-1}, and $N_p=N_{3p}=0$ by \eqref{p2pzero}. Thus
\be\label{f3} N_{p-1}-2N_{2p-1}+3N_{3p-1}\equiv -2\mod p.\ee

The equations \eqref{f1}, \eqref{f2} and \eqref{f3} form a system of $3$ linear equations in the variables $N_{p-1}$, $N_{2p-1}$ and $N_{3p-1}$ over the field $\Z/p\Z$. However, for $p\geq 3$, basic linear algebra shows this system to be inconsistent, yielding a proof concluding contradiction.
\end{proof}

The remainder of the section is devoted to the constant $\mathsf s_{k\exp(G)}(G)$. We begin with the refinement to the result obtained via R\'onyai's method.

\begin{proof}[Proof of Theorem \ref{thm-alg-expconsequence}]
Letting $X'\subseteq X$ be the subset consisting of the smallest $d+m$ elements in $X$, we have $\max X'\leq \max X-(|X|-d-m)$.
Since $\max X'<\min (X\setminus X')$, it follows that \eqref{det-hyp} also holds for $X'$.
If the result holds whenever $|X|=d+m$, then applying this case to $X'$ yields \begin{align*}\mathsf s_{X\cdot q}(G)\leq s_{X'\cdot q}(G)&\leq \big(\max X'-d-\frac{m}{p}+1\big)q+\mathsf D^*(G)-1\\&\leq \big(\max X-|X|+\frac{m(p-1)}{p}+1\big)q+\mathsf D^*(G)-1\\
&\leq\big(\max X+1-\frac{m}{p}\big)q-r,\end{align*} with the third inequality in view of the hypothesis $|X|\geq d+m$, as desired. Therefore it suffices to handle the case when $|X|=d+m$, which we now assume. We need to show
$$\mathsf s_{ X\cdot q}(G)\leq \big(k-d-\frac{m}{p}+1\big)q+\mathsf D^*(G)-1, $$ where $ k=\max X$. Let $\{x_1,\ldots,x_s\}=[1,k]\setminus X$, where $s=k-d-m$ and $1\leq x_1<\ldots<x_s<k$.


Write $G=C_{q_1}\oplus \ldots\oplus C_{q_r}$ with each $q_i$ a power of $p$ and $$1<q_1\leq \ldots\leq q_r=q_{r+1}:=q.$$ Then $\mathsf D^*(G)=\Sum{i=1}{r}(q_i-1)+1$. Let $(e_1,\ldots,e_r)$ be a basis for $G$ with $\ord(e_i)=q_i$ for $i\in[1,r]$.
Let $S=g_1\bdot \ldots\bdot g_\ell$ be a sequence of terms from $G$ with $|S|=\ell=(k-d-\frac{m}{p}+1)q+\mathsf D^*(G)-1$.
We have  \be\label{k-est}\left\lfloor\frac{|S|}{q}\right\rfloor\leq k-d+\left\lfloor 1+\frac{\mathsf D^*(G)-1}{q}\right\rfloor=k.\ee For each $i\in [1,\ell]$, write $$g_i=\Sum{j=1}{r}a_i^{(j)}e_j\quad\mbox{ with $a_i^{(j)}\in [0,q_j-1]$}. $$

Let $$f_j(\textbf x)=\Sum{i=1}{\ell}a_i^{(j)}X_i^{p-1}\in \Z[X_1,\ldots,X_\ell], \quad\mbox{ for $j\in[1,r]$}.$$
Let $$f_{r+1}(\textbf x)=\Sum{i=1}{\ell}X_i^{p-1}\in \Z[X_1,\ldots,X_\ell].$$
For $i\in [0,k-d-m]$, let $$w_i(X)=X^i\in \Z[X].$$
In view of Proposition \ref{prop-Ax-Katz-ChooseI}, let $\mathcal I\subseteq [0,qp^{m+1}-1]$ be a complete system of residues modulo $p$ such that \be\label{fermat-inproof}x^{p-1}\equiv
\left\{
  \begin{array}{ll}
    1 \mod qp^{m+1}& \hbox{if $x\not\equiv 0\mod p$} \\
    0\mod qp^{m+1} & \hbox{if $x\equiv 0\mod p$,}
  \end{array}
\right.\quad\mbox{ for every $x\in \mathcal I$}.\ee
Observe that  $\max_{j\in [1,r+1]}\big\{1,\;\frac{\varphi(q_j)}{p-1}\deg f_j\big\}=\max_{j\in [1,r+1]}\big\{\varphi(q_j)\big\}=\varphi(q)=(p-1)\frac{q}{p}$ and
\begin{align*}\ell=|S|&=m(p-1)\frac{q}{p}+
\Sum{j=1}{r}(q_j-1)+(k-d-m+1)q\\
&=m\max_{j\in [1,r+1]}\big\{1,\;\frac{\varphi(q_j)}{p-1}\deg f_j\big\}+\Sum{j=1}{r}\frac{q_j-1}{p-1}\deg f_j+\frac{(k-d-m+1)q-1}{p-1}\deg f_{r+1}+1.\end{align*} Thus, for each $i\in [0,k-d-m]$, we can apply Theorem  \ref{thm-axkatz-gen}, with $m$ taken to be $m+1$,  taking $\mathcal I_j=\mathcal I$ for all $j$, and using the polynomials $f_1,\ldots,f_r,f_{r+1}$, weights $\underbrace{w_0,\ldots,w_0}_r,w_i$, and prime powers $q_1,\ldots,q_r,q$.
As a result, letting $$V=\{\textbf a\in \mathcal I^\ell:\; f_j(\textbf a)\equiv 0\mod q_j\;\mbox{ for all $j\in [1,r+1]$}\},$$
it follows that the weighted size of $V$ is congruent to $0$ modulo $p^{m+1}$, for each $i\in [0,k-d-m]$. Let  us next describe what this size equals.

Let $$N_j:=N_{jq}(S)\quad\mbox{ for $j\in [0,k]$}.$$ Let $i\in [0,k-d-m]$ be arbitrary. Associate to each $\textbf a\in \mathcal I^\ell$ the subsequence $S_\textbf a=\prod_{j\in I_\textbf a}^\bullet g_j$, where $I_\textbf a\subseteq [1,\ell]$ consists of all $j\in [1,\ell]$ for which the $j$-th coordinate of $\textbf a$ is nonzero modulo $p$. In view of \eqref{fermat-inproof}, the conditions $f_j(\textbf a)\equiv 0\mod q_j$, for $j\in[1,r]$, restrict to tuples $\textbf a\in \mathcal I^\ell$ for which the associated sequence $S_\textbf a$ is zero-sum. Likewise, the additional condition $f_{r+1}(\textbf a)\equiv 0\mod q$ further restricts to tuples $\textbf a\in \mathcal I^\ell$ whose associated sequence $S_\textbf a$ has length $|S_\textbf a|=|I_\textbf a|\equiv 0\mod q$. This means that  the tuples $\textbf a\in V$ are precisely those whose associated sequence $S_\textbf a$ is a zero-sum subsequence of length $|S_\textbf a|\equiv 0\mod q$, meaning $|S_\textbf a|=jq$ for some $j\in[0,k]$ (in view of \eqref{k-est}).
Moreover, each zero-sum subsequence of length $jq$ is associated to exactly $(p-1)^{jq}$ tuples $\textbf a\in \mathcal I^\ell$, and the weighted size of each such tuple is $w_i(j)\equiv j^i\mod p^{m+1}$ (in view of \eqref{fermat-inproof}). As a result, for $i\in [0,k-d-m]$, the weighted size of $V$ equals $\Sum{j=0}{k}j^i(p-1)^{jq}N_j\mod p^{m+1}$, meaning the conclusion of Theorem \ref{thm-axkatz-gen} is that
$$(p-1)^{q}N_1+(p-1)^{2q}N_2+\ldots +(p-1)^{jq}N_j+\ldots+(p-1)^{kq}N_k\equiv -N_0=-1 \mod p^{m+1}$$ and $$(p-1)^{q}N_1+2^i(p-1)^{2q}N_2+\ldots +j^i(p-1)^{jq}N_j+\ldots+k^i(p-1)^{kq}N_k\equiv 0 \mod p^{m+1},$$ for every  $i\in [1,k-d-m]$.

Assuming by contradiction that $S$ has no zero-sum subsequence of length $kq$ with $k\in X$, it follows that $N_j=0$ for all $j\in X$. This leaves us with a system of $k-d-m+1$ linear equations, one for each $i\in [0,k-d-m]$,  in the $k-d-m$ variables $N_j$, where $j\in [1,k]\setminus X$, over the ring $R=\Z/p^{m+1}\Z$. We proceed to show this system is inconsistent, which will complete the proof.

Let $A'$ be $(k-d-m+1)\times (k-d-m)$ matrix, with rows indexed by $i\in [0,k-d-m]$, columns indexed by $j\in [1,k]\setminus X$, and $(i,j)$-th entry equal to $j^i(p-1)^{jq}$, and let $\textbf y$ be the column vector $[N_j]_{j\in [1,k]\setminus X}$. Then the above system of linear equations
can be written as $A'\textbf y\equiv [-1,0,\ldots,0]\mod p^{m+1}$. To show this system is inconsistent, it suffices to show a nonzero (modulo $p^{m+1}$) multiple  the first row of $A'$ can be written as a linear combination of the remaining rows. To this end, let $A=[j^i(p-1)^{jq}]_{i\in [1,k-d-m],j\in [1,k]\setminus X}$ be the $(k-d-m)\times (k-d-m)$ matrix obtained from $A'$ by removing the first row. We continue by calculating $\det A$. Note that $A$ can be obtained from the matrix $B=[j^i]_{i\in [0,k-d-m-1],j\in [1,k]\setminus X}$ by multiplying each  $j$-th column of $B$ by $j(p-1)^{jq}$. Thus $$\det A=\Big(\prod_{j\in [1,k]\setminus X}j(p-1)^{jq}\Big)\det B=\Big(\prod_{j=1}^{s}x_j(p-1)^{x_jq}\Big)\det B,$$ where we recall that $[1,k]\setminus X=\{x_1,\ldots,x_s\}$ with $x_1<\ldots<x_s$ (by hypothesis). However, note that $B$ is simply a Vandermonde matrix, whose well-known determinant (see \cite{garrett-algtext}) equals $\det B=\prod_{1\leq i<j\leq s}(x_j-x_i)$. It follows that
$$\det A=\Big(\prod_{j=1}^{s}x_j(p-1)^{x_jq}\Big)\Big(\prod_{1\leq i<j\leq s}(x_j-x_i)\Big)\not\equiv0\mod p^{m+1},$$ with this determinant being nonzero by hypothesis. In consequence,  the rows of $A$ are linearly independent over $\Q$, meaning there is some $\Q$-linear combination of the rows of $A$ equal to the first row in $A'$.
Moreover, since the entries of $A'$ are integers, Cramer's Rule (see \cite{garrett-algtext}) ensures that each coefficient in this linear combination has it denominator dividing $\det A$. By clearing denominators, it then follows that there is a $\Z$-linear combination of the rows of $A$ equal to  the first row of $A'$ multiplied by the integer  $\det A\not\equiv 0\mod p^{m+1}$. Reducing modulo $p^{m+1}$, we obtain a linear combination of the rows of $A$ equal to a nonzero (modulo $p^{m+1}$) multiple of the first row of $A'$, which shows that the system of linear equations is inconsistent, completing the proof as noted earlier.
\end{proof}

The following is the main step in the proof of Theorem \ref{thm-sk-gen}.

\begin{proposition}\label{prop-large-d}
Let $G$ be a finite abelian $p$-group with exponent $q$,  let $d=\left \lceil \frac{\mathsf D^*(G)}{q}\right\rceil$, and let $k$ be an integer such that $\frac{d(d-1)}{2}+1\leq k\leq p.$
Then $$\mathsf s_{kq}(G)\leq kq+\mathsf D^*(G)-1.$$
\end{proposition}

\begin{proof} If $q=1$, then $G$ is trivial with $\mathsf s_{kq}(G)=kq=kq+\mathsf D^*(G)-1$, as desired. Therefore we can assume $q>1$.
Let $r\in [1,q]$ be the integer such that $d=\frac{\mathsf D^*(G)+r-1}{q}$. Note that  $d\geq 1$.
Assume by contradiction that $S$ is a sequence of terms from $G$ with $$0\notin \Sigma_{kq}(S)\quad\und\quad |S|=kq+\mathsf D^*(G)-1=(k+d)q-r.$$

\textbf{Claim A:} There are disjoint subsequences $T_1\bdot\ldots\bdot T_{d-1}\mid S$ such that each $T_i$ is zero-sum with $|T_i|=iq$, for every $i\in [1,d-1]$.

\begin{proof}
Let $Y\subseteq [1,d-1]$ be a maximal subset (possibly empty) such that there are disjoint subsequences $\prod^\bullet_{i\in Y}T_i\mid S$ with each  $T_i$ is zero-sum  and $|T_i|=iq$, for every $i\in Y$. To establish the claim, we need to show $Y=[1,d-1]$. If $d=1$, then the claim is trivial taking $Y=\emptyset$, so we can assume $d\geq 2$.

We begin by showing $|Y|\geq 1$. To this end, let $X=[1,d-1]\cup \{k\}$. In view of $k\geq \frac{d(d-1)}{2}+1\geq d\geq 1$, we have $X\subseteq \mathbb N$ and  $|X|=d$. In view of $k\leq p$, we have $[1,\max X]\setminus X=[d,k-1]\subseteq [d,p-1]$. Thus, since $|S|=(k+d)q-r\geq (k+1)q-r$ (as $d\geq 1$), we can apply Theorem \ref{thm-alg-expconsequence} with $X=[1,d-1]\cup\{k\}$ and $m=0$ to conclude that there is some zero-sum subsequence $T\mid S$ with $|T|\in ([1,d-1]\cup \{k\})\cdot q$. Since $0\notin \Sigma_{kq}(S)$, it thus follows that $|T|=iq$ for some $i\in [1,d-1]$, and taking $T_i=T$ and $Y=\{i\}$ now shows that $|Y|\geq 1$. The claim is now complete unless $d\geq 3$.

We continue by showing that $|Y|\geq 2$. If this fails, then we have $Y=\{y_1\}$ for some $y_1\in [1,d-1]$, and there is a zero-sum subsequence $T_1\mid S$ with $|T_1|=y_1q$. Since $0\notin \Sigma_{kq}(S)$, we have
\be\label{A-1}0\notin \Sigma_{\{(k-y_1),k\}\cdot q}(T_1^{[-1]}\bdot S).\ee Let $X=\big([1,d-1]\setminus\{y_1\}\big)\cup \{k-y_1\}\cup \{k\}$. Since $k\geq \frac{d(d-1)}{2}+1\geq 2(d-1)$  and $y_1\in [1,d-1]$, we have $X\subseteq \mathbb N$ with $|X|=d$. Since $k\leq p$, we have $[1,\max X]\setminus X\subseteq [1,k-1]\subseteq [1,p-1]$. Since $y_1\leq d-1$, we have $|T_1^{[-1]}\bdot S|=(k-y_1+d)q-r\geq (k+1)q-r$. As a result, we can apply Theorem \ref{thm-alg-expconsequence} to $T_1^{[-1]}\bdot S$ with $X=\big([1,d-1]\setminus\{y_1\}\big)\cup \{k-y_1\}\cup \{k\}$ and $m=0$ to find a zero-sum subsequence $T_2\mid T_1^{[-1]}\bdot S$ with $|T_2|=y_2q$ for some $y_2\in [1,d-1]\setminus \{y_1\}$ (in view of \eqref{A-1}). But now the set $\{y_1, y_2\}$ can be taken for $Y$, showing that  $|Y|\geq 2$. The claim is now complete unless $d\geq 4$.

In view of the our prior work, let $s:=|Y|\geq 2$, let $Y=\{y_1,\ldots,y_s\}$, and let $T_1\bdot\ldots\bdot T_s\mid S$ with each $T_i$ a zero-sum subsequence of length $|T_i|=y_iq$ with $y_i\in [1,d-1]$, for every $i\in [1,s]$. Assume by contradiction that $2\leq s\leq d-2$. Let $y=y_1+\ldots+y_s$ and let $$\max( [1,d-1]\setminus Y)=d-s_0,\quad\mbox{ where $s_0\in [1,s+1]$}.$$  Observe that \begin{align}\label{Abus} y&\leq \Sum{i=1}{s+1}(d-i)-(d-s_0)=\frac{s(2d-s-3)}{2}+s_0-1\leq \frac{d(d-1)}{2}-1\leq k-2,\end{align} with the final inequality holding by hypothesis.
Let $T^*=y_1\bdot\ldots\bdot y_s$, which is a sequence of terms from $\Z$.
Since $0\notin \Sigma_{kq}( S)$, we have \be\label{Bbus}0\notin \Sigma_{(k-t)q}\big((T_1\bdot\ldots\bdot T_s)^{[-1]}\bdot S\big), \quad\mbox{for every $t\in \Sigma(T^*)\cap [1,k-1]$}.\ee
Since $s\geq 2$, we have  $$y\in \Sigma(T^*)\quad\und\quad y-y_i=y_1+\ldots+y_{i-1}+y_{i+1}+\ldots+y_s\in\Sigma(T^*),\quad\mbox{ for every $i\in [1,s]$}.$$
Hence, in view of \eqref{Abus} and $y_1,\ldots,y_s\geq 1$, it follows that $y,y-y_1,\ldots,y-y_s\in \Sigma(T^*)\cap [1,k-1]$ are distinct elements. Thus \eqref{Bbus} implies that \be\label{CCbus}0\notin \Sigma_{\{(k-y),(k-y+y_1),\ldots,(k-y+y_s)\}\cdot q}\big((T_1\bdot\ldots\bdot T_s)^{[-1]}\bdot S\big).\ee

Now let $X=\big([1,d-1]\setminus\{y_1,\ldots,y_s\}\big)\cup  \{k-y,\, k-y+y_1,\,\ldots,\, k-y+y_s\}$. By definition of $s_0$, we have $\max \big([1,d-1]\setminus\{y_1,\ldots,y_s\}\big)=d-s_0$. If $k-y\leq d-s_0$, then \eqref{Abus} and  $s\leq d-2$ yield  $$k\leq d-s_0+y\leq d+\frac{s(2d-s-3)}{2}-1\leq \frac{d(d-1)}{2},$$   contrary to  hypothesis. Therefore, we must instead have $k-y> d-s_0$, which ensures that
\begin{align*}\max \big([1,d-1]\setminus\{y_1,\ldots,y_s\}\big)&<\min\Big(\{k-y,\, k-y+y_1,\,\ldots,\, k-y+y_s\}\Big)\quad \und\\
&|X|=d.\end{align*}
In view of $k\leq p$, we have $[1,\max X]\setminus X\subseteq [1,k-2]\subseteq [1,p-2]$. We also have $|(T_1\bdot\ldots\bdot T_s)^{[-1]}\bdot S|=|S|-yq=(k-y+d)q-r$. As  a result, in view of $y_1,\ldots,y_s\in [1,d-1]$, it follows that we can apply Theorem \ref{thm-alg-expconsequence}  using $m=0$ and $$X=\big([1,d-1]\setminus\{y_1,\ldots,y_s\}\big)\cup  \{k-y,\, k-y+y_1,\,\ldots,\, k-y+y_s\}$$  to conclude in view of \eqref{CCbus} that there is a zero-sum subsequence $T_{s+1}\mid (T_1\bdot\ldots\bdot T_s)^{[-1]}\bdot S$ with $|T_{s+1}|=y_{s+1}q$ for some $y_{s+1}\in [1,d-1]\setminus Y=[1,d-1]\setminus \{y_1,\ldots,y_s\}$. But now $\{y_1,\ldots,y_s,y_{s+1}\}$ contradicts the maximality of $Y$, completing the proof of the claim.
\end{proof}

Let $y=\frac{d(d-1)}{2}=\Summ{i\in [1,d-1]}i$, and let $X=[k-y,k-y+d-1]$. Since $k\geq \frac{d(d-1)}{2}+1$ by hypothesis, we have $X\subseteq \mathbb N$ and $|X|=d$. Since $k\leq p$, we have $[1,\max X]\setminus X=[1,k-y-1]\subseteq [1,p-1]$. In view of Claim A, we have $|(T_1\bdot\ldots\bdot T_{d-1})^{[-1]}\bdot S|=|S|-yq=(k-y+d)q-r$. As a result, we can apply Theorem \ref{thm-alg-expconsequence} to $(T_1\bdot\ldots\bdot T_{d-1})^{[-1]}\bdot S$ with $X=[k-y,k-y+d-1]$ and $m=0$ to conclude that
\be\label{tagyourit}0\in \Sigma_{[(k-y),k-y+d-1]\cdot q}\big((T_1\bdot\ldots\bdot T_{d-1})^{[-1]}\bdot S\big).\ee In view of Claim A, we have $0\in \Sigma_{tq}(T_1\bdot\ldots\bdot T_{d-1})$ for every $t\in [0,y]$, which combined with \eqref{tagyourit} implies that $0\in \Sigma_{kq}(S)$, contrary to assumption, completing  the proof.
\end{proof}

Next, we handle the main step in the proof of Theorem \ref{thm-sk-spec}.

\begin{proposition}\label{prop-small-d}
Let $G$ be a finite abelian $p$-group with exponent $q$,  let $d=\left \lceil \frac{\mathsf D^*(G)}{q}\right\rceil$. Suppose  $d\leq 4$ and $k$ is an integer with  $d\leq k\leq p$.  Then $$\mathsf s_{kq}(G)\leq kq+\mathsf D^*(G)-1.$$
\end{proposition}

\begin{proof}If $q=1$, then $G$ is trivial with $\mathsf s_{kq}(G)=kq=kq+\mathsf D^*(G)-1$, as desired. Therefore we can assume $q>1$.
Note that $d\geq 1$. Assume by contradiction that $S$ is a sequence of terms from $G$ with $$0\notin \Sigma_{kq}(S)\quad\und\quad |S|=kq+\mathsf D^*(G)-1=(k+d)q-r,$$  where  $r\in [1,q]$ is the integer such that $d=\frac{\mathsf D^*(G)+r-1}{q}$.

\textbf{Case 1:} $d=1$

Let $X=\{k\}$.
Since $1=d\leq k\leq p$, we have $X\subseteq \mathbb N$ and  $[1,\max X]\setminus X=[1,k-1]\subseteq [1,p-1]$, allowing us to  apply Theorem \ref{thm-alg-expconsequence} using $X=\{k\}$ and $m=0$ to conclude that $\mathsf s_{kq}(G)\leq kq+\mathsf D^*(G)-1$, as desired.

\textbf{Case 2:} $d=2$

Note that $k\geq d=2$.
Suppose  there is a zero-sum subsequence $T\mid S$ with $|T|=q$. Then $0\notin \Sigma_{kq}(S)$ ensures that
$0\notin \Sigma_{\{(k-1),k\}\cdot q}(T^{[-1]}\bdot S)$. Let  $X=\{k-1,k\}$.
  In view of $k\geq 2$, we have $X\subseteq \mathbb N$ and $|X|=2$. In view $k\leq p$, we have $[1,\max X]\setminus X=[1,k-2]\subseteq [1,p-2]$ and $|T^{[-1]}\bdot S|=(k+1)q-r$, allowing us to  apply Theorem \ref{thm-alg-expconsequence} to $T^{[-1]}\bdot S$ using $X=\{k,k-1\}$ and $m=0$ to conclude that $0\in\Sigma_{\{(k-2),k\}\cdot q}(T^{[-1]}\bdot S)$, contradicting that the opposite  was just shown. So we instead conclude that \be\label{0notin}0\notin \Sigma_{\{1,k\}\cdot q}(S).\ee
 Now let $X=\{1,k\}$. In view of $k\geq 2$, we have $X\subseteq \mathbb N$ and $|X|=2$. In view of $k\leq p$, we have $[1,\max X]\setminus X=[2,k-1]\subseteq [1,p-1]$ and $|S|\geq (k+1)q-r$, allowing us to apply Theorem \ref{thm-alg-expconsequence} to $S$ using $X=\{1,k\}$ and $m=0$ to conclude that $0\in \Sigma_{\{1,k\}\cdot q}(S)$, contrary to
\eqref{0notin}.

\textbf{Case 3:} $d=3$

Note that $k\geq d=3$. Suppose there is a zero-sum subsequence $T_1\mid S$ with $|T_1|=q$.  Then $0\notin \Sigma_{kq}(S)$ ensures that
$0\notin \Sigma_{\{(k-1),k\}\cdot q}(T_1^{[-1]}\bdot S)$. Let $X=\{1,k-1,k\}$. In view of $k\geq d=3$, we have $X\subseteq \mathbb N$ with $|X|=3$.  In view of  $k\leq p$, we have $[1,\max X]\setminus X=[2,k-2]\subseteq [2,p-2]$ and $|T_1^{[-1]}\bdot S|=(k+2)q-r$, allowing us to  apply Theorem \ref{thm-alg-expconsequence} using $X=\{1,k-1,k\}$ and $m=0$ to conclude that $0\in \Sigma_{\{1,(k-1),k\}\cdot q}(T_1^{[-1]}\bdot S)$, which in view of $0\notin \Sigma_{\{(k-1),k\}\cdot q}(T_1^{[-1]}\bdot S)$ means there is some zero-sum subsequence $T_2\mid T_1^{[-1]}\bdot S$ with $|T_2|=q$. But now
$0\notin \Sigma_{kq}(S)$  ensures that
$$0\notin \Sigma_{\{(k-2),(k-1),k\}\cdot q}(T_1^{[-1]}\bdot T_2^{[-1]}\bdot  S).$$ Now let $X=\{k-2,k-1,k\}$. Note $X\subseteq \mathbb N$ with $|X|=3=d$ in view of $k\geq d=3$. In view of $k\leq p$, we have $[1,\max X]\setminus X=[1,k-3]\subseteq [1,p-3]$ and $|T_1^{[-1]}\bdot T_2^{[-1]}\bdot  S|=(k+1)q-r$, allowing us to  apply Theorem \ref{thm-alg-expconsequence} using $X=\{k-2,k-1,k\}$ and $m=0$ to conclude that $0\in \Sigma_{\{(k-2),(k-1),k\}\cdot q}(T_1^{[-1]}\bdot T_2^{[-1]}\bdot  S)$, contrary to what was just noted.
So we instead conclude that \be\label{0notinb}0\notin \Sigma_{\{1,k\}\cdot q}(S).\ee

Suppose there is a zero-sum subsequence $T\mid S$ with $|T|=(k+2)q$. Let $X=\{1,k,k+1\}$. Then $X\subseteq \mathbb N$ with $|X|=3$ in view of $k\geq 2$. Since the complement of a zero-sum subsequence in $T$ is also zero-sum, we conclude from \eqref{0notinb} that $0\notin \Sigma_{\{1, k,(k+1)\}\cdot q}(T)$. In view of $k\leq p$, we have
$[1,\max X]\setminus X=[2,k-1]\subseteq [2,p-1]$ and $|T|=(k+2)q\geq (k+2)q-r$, allowing us to  apply Theorem \ref{thm-alg-expconsequence} using $X=\{1,k,k+1\}$ and $m=0$ to conclude that $0\in \Sigma_{\{1,k,(k+1)\}\cdot q}(T)$, contrary to what was just noted. So we instead conclude that \be\label{0notinc}0\notin \Sigma_{\{1,k,(k+2)\}\cdot q}(S).\ee

Now let $X=\{1,k,k+2\}$. Then $|S|=(k+3)q-r$ and $[1,\max X]\setminus X=[2,k-1]\cup \{k+1\}$. We also have  $k\leq p$. As a result, unless $p=k+1$, we can  apply Theorem \ref{thm-alg-expconsequence} using $X=\{1,k,k+2\}$ and $m=0$ to conclude that $0\in \Sigma_{\{1,k,(k+2)\}\cdot q}(S)$, contrary to \eqref{0notinc}. Therefore we must have $p=k+1\geq d+1=4$, whence $k+1=p\geq 5$ as $p$ is prime. In particular, $k\geq 4$.

Now let $X=\{1,2,k\}$. Note that $|X|=3$ in view of $k\geq 3$. In view of $k\leq p$, we have  $[1,\max X]\setminus X=[3,k-1]\subseteq [3,p-1]$ and $|S|=(k+3)q-r$, allowing us to apply Theorem \ref{thm-alg-expconsequence} using $X=\{1,2,k\}$ and $m=0$ to conclude that $0\in \Sigma_{\{1,2,k\}\cdot q}(S)$, which in view of \eqref{0notinc} implies that there is a zero-sum subsequence $T\mid S$ with $|T|=2q$.
But now $0\notin \Sigma_{kq}(S)$ ensures that $0\notin \Sigma_{(k-2)q}(T^{[-1]}\bdot S)$.
Thus \eqref{0notinc} yields \be\label{notind}0\notin \Sigma_{\{1,(k-2),k\}\cdot q}(T^{[-1]}\bdot S).\ee

Now let $X=\{1,k-2,k\}$. In view of $k\geq 4$, we have $X\subseteq \mathbb N$ and  $|X|=3$. In view of $k\leq p$, we have  $[1,\max X]\setminus X\subseteq [2,k-1]\subseteq [2,p-1]$ and $|T^{[-1]}\bdot S|=(k+1)q-r$, allowing us to apply Theorem \ref{thm-alg-expconsequence} using $X=\{1,k-2,k\}$ and $m=0$ to conclude that $0\in \Sigma_{\{1,(k-2),k\}\cdot q}(T^{[-1]}\bdot S)$, contrary to \eqref{notind}.

\textbf{Case 4:} $d=4$.

Note that $k\geq d=4$. We divide the proof into five subcases. Note, since $p$ is prime, that $k=5$ and $p=k+1$ cannot both hold, ensuring all possibilities are covered.

\smallskip

CASE 4.1: $0\notin \Sigma_{\{1,2\}\cdot q}(S)$.

Suppose there is a  zero-sum subsequence $T\mid S$ with $|T|=(k+1)q$. Then, since the complement of zero-sum subsequence of $T$ is also zero-sum, it follows from the subcase hypothesis $0\notin \Sigma_{\{1,2\}\cdot q}(S)$ that  $0\notin \Sigma_{\{1,2,(k-1),k\}\cdot q}(T)$. Let $X=\{1,2,k-1,k\}$. Since $k\geq 4$, we have $X\subseteq \mathbb N$ and $|X|=4$. In view of $k\leq p$, we have $[1,\max X]\setminus X=[3,k-2]\subseteq [3,p-2]$ and $|T|=(k+1)q\geq (k+1)q-r$, allowing us to apply Theorem \ref{thm-alg-expconsequence} to $T$ with $X=\{1,2,k-1,k\}$ and $m=0$ to conclude that $0\in\Sigma_{\{1,2,(k-1),k\}\cdot q}(T)$, contrary to what was just noted.
So we instead conclude that \be\label{vultA}0\notin \Sigma_{\{1,2,k,(k+1)\}\cdot q}(S).\ee
 Now let $X=\{1,2,k,k+1\}$. Since $k\geq 3$, we have $X\subseteq \mathbb N$ and $|X|=4$. In view of $k\leq p$, we have $[1,\max X]\setminus X=[3,k-1]\subseteq [3,p-1]$ and $|S|=(k+4)q-r$. But now Theorem \ref{thm-alg-expconsequence} applied to $S$ with $X=\{1,2,k,k+1\}$ and $m=0$ yields $0\in\Sigma_{\{1,2,k,(k+1)\}\cdot q}(S)$, contrary to \eqref{vultA}.

\smallskip

CASE 4.2: There exists disjoint subsequences $T_1\bdot T_2\mid S$ with $|T_1|=q$, $|T_2|=2q$, and $T_1$ and $T_2$ each zero-sum.

In this case, there are zero-sum subsequences of $T_1\bdot T_2$ having lengths $q$, $2q$ and also $3q$. As a result, since $0\notin \Sigma_{kq}(S)$, we have $$0\notin \Sigma_{\{(k-3),(k-2),(k-1),k\}\cdot q}(T_1^{[-1]}\bdot T_2^{[-1]}\bdot S).$$
Let $X=[k-3,k]$. Since $k\geq d=4$, we have $X\subseteq \mathbb N$ and $|X|=4$. Since $k\leq p$, we have $[1,\max X]\setminus X=[1,k-4]\subseteq [1,p-4]$. Hence, since $|T_1^{[-1]}\bdot T_2^{[-1]}\bdot S|=(k+1)q-r$, we can apply Theorem \ref{thm-alg-expconsequence} to $T_1^{[-1]}\bdot T_2^{[-1]}\bdot S$ with $X=[k-3,k]$ and $m=0$ to conclude that $0\in\Sigma_{\{(k-3),(k-2),(k-1),k\}\cdot q}(T_1^{[-1]}\bdot T_2^{[-1]}\bdot S)$, contrary to what was just noted.

\smallskip

CASE 4.3: $0\in \Sigma_{q}(S)$.

Let $T_1\mid S$ be a zero-sum subsequence with $|T_1|=q$, which exists by subcase hypothesis.
In view of  CASE 4.2 and $0\notin \Sigma_{kq}(S)$, we can assume  \be\label{catch}0\notin \Sigma_{\{2,(k-1),k\}\cdot q}(T_1^{[-1]}\bdot S).\ee

Suppose there is a zero-sum subsequence $T\mid T_1^{[-1]}\bdot S$ with $|T|=(k+1)q$. Then, since the complement of a zero-sum subsequence of $T$ is also zero-sum, it follows from \eqref{catch} that $0\notin\Sigma_{\{1, 2,(k-1),k\}}(T)$. Let $X=\{1,2,k-1,k\}$. Since $k\geq d= 4$, we have $X\subseteq \mathbb N$ and $|X|=4$. Since $p\geq k$, we have $[1,\max X]\setminus X=[3,k-2]\subseteq [3,p-2]$. Hence, since $|T|=(k+1)q\geq (k+1)q-r$, we can apply Theorem \ref{thm-alg-expconsequence} to $T$ with $X=\{1,2,k-1,k\}$ and $m=0$ to conclude that $0\in\Sigma_{\{1,2,(k-1),k\}\cdot q}(T)$, contrary to what was just noted. So we instead conclude that $0\notin \Sigma_{(k+1)q}(T_1^{[-1]}\bdot S)$, which along with \eqref{catch} ensures that  \be\label{catchB}0\notin \Sigma_{\{2,(k-1),k,(k+1)\}\cdot q}(T_1^{[-1]}\bdot S).\ee

Now let $X=\{2,k-1,k,k+1\}$. Since $k\geq d=4$, we have $X\subseteq \mathbb N$ and $|X|=4$. Since $p\geq k$, we have $[1,\max X]\setminus X=\{1\}\cup [3,k-2]\subseteq [1,p-2]$. Hence, since $|T_1^{[-1]}\bdot S|=(k+3)q-r$, we can apply Theorem \ref{thm-alg-expconsequence} to $T_1^{[-1]}\bdot S$ with $X=\{2,k-1,k,k+1\}$ and $m=0$ to conclude that $0\in\Sigma_{\{2,(k-1),k,(k+1)\}\cdot q}(T)$, contrary to  \eqref{catchB}.

\smallskip

CASE 4.4: $p\neq k+1$.

We have $0\notin\Sigma_{kq}(S)$ and can assume $0\notin \Sigma_{q}(S)$ in view of CASE 4.3.

Suppose there is a zero-sum subsequence $T\mid S$ with $|T|=tq$ for some $t\in [k+2,k+3]$. Then, since $0\notin \Sigma_{\{1,k\}\cdot q}(S)$ with the complement of a zero-sum subsequence in $T$ also zero-sum, it follows that \be\nn 0\notin \Sigma_{\{1,(t-k), k,(t-1)\}\cdot q}(T).\ee Let $X=\{1,t-k,k,t-1\}$. Since $k\geq d=4$ and $k+2\leq t\leq k+3<2k$, we have $X\subseteq \mathbb N$ and $|X|=4$. If $t=k+2$, then  $[1,\max X]\setminus X=[3,k-1]$. If $t=k+3$, then $[1,\max X]\setminus X=\{2\}\cup [4,k-1]\cup \{k+1\}$. In either case, since $p\geq k$ with $p\neq k+1$ (by subcase hypothesis), it follows in view of $|T|=tq\geq tq-r$ that we can apply Theorem \ref{thm-alg-expconsequence} to $T$ with $X=\{1,t-k,k,t-1\}$ and $m=0$ to conclude that $0\in\Sigma_{\{1,(t-k),k,(t-1)\}\cdot q}(T)$, contrary to  what was noted above. So we instead conclude that $0\notin \Sigma_{\{(k+2),(k+3)\}\cdot q}(S)$, which along with the already noted fact that $0\notin \Sigma_{\{1,k\}\cdot q}(S)$ means
\be\label{stuffilb}0\notin \Sigma_{\{1,k,(k+2),(k+3)\}\cdot q}(S).\ee

Now let $X=\{1,k,k+2,k+3\}$. Since $k\geq d=4$, we have $X\subseteq \mathbb N$ and $|X|=4$. Since $p\geq k$, we have $[1,\max X]\setminus X=[2,k-1]\cup \{k+1\}$. We also have $p\geq k$ with $p\neq k+1$ by subcase hypothesis. Hence, since $|S|=(k+4)q-r$, we can  apply Theorem \ref{thm-alg-expconsequence} to $S$ with $X=\{1,k,k+2,k+3\}$ and $m=0$ to conclude that $0\in\Sigma_{\{1,k,(k+2),(k+3)\}\cdot q}(S)$, contrary to  \eqref{stuffilb}.

\smallskip

CASE 4.5: $k\neq 5$.

In view of CASES 4.1 and 4.3, we can assume there is a zero-sum subsequence $T_2\mid S$ with $|T_2|=2q$. Then, since $0\notin \Sigma_{kq}(S)$, it follows in view of CASE 4.3 that \be\label{tiltA}0\notin \Sigma_{\{1,(k-2),k\}\cdot q}(T_2^{[-1]}\bdot S).\ee

Suppose there is a zero-sum subsequence $T\mid T_2^{[-1]}\bdot S$ with $|T|=(k+1)q$. Then, since the complement of a zero-sum subsequence in $T$ is also zero-sum, it follows from \eqref{tiltA} that $$0\notin\Sigma_{\{1,3,(k-2),k\}\cdot q}(T).$$ Let $X=\{1,3,k-2,k\}$. Since $k\geq d=4$ and $k\neq 5$ (in view of the subcase hypothesis), we have $X\subseteq \mathbb N$ and $|X|=4$. Since $p\geq k$, we have $[1,\max X]\setminus X\subseteq [2,k-1]\subseteq [2,p-1]$. Hence, since $|T|=(k+1)q\geq (k+1)q-r$, we can apply Theorem \ref{thm-alg-expconsequence} to $T$ with $X=\{1,3,k-2,k\}$ and $m=0$ to conclude that $0\in\Sigma_{\{1,3,(k-2),k\}\cdot q}(T)$, contrary to  what was noted above.
So we can now assume $0\notin \Sigma_{(k+1)q}(T_2^{[-1]}\bdot S)$, which together with \eqref{tiltA} means
\be\label{tiltB}0\notin \Sigma_{\{1,(k-2),k,(k+1)\}\cdot q}(T_2^{[-1]}\bdot S).\ee

Now let $X=\{1,k-2,k,k+1\}$. In view of $k\geq d=4$, we have $X\subseteq \mathbb N$ and $|X|=4$. In view of $p\geq k$, we have $[1,\max X]\setminus X=[2,k-3]\cup \{k-1\}\subseteq [2,p-1]$. Hence, since $|T_2^{[-1]}\bdot S|=(k+2)q-r$, we can apply Theorem \ref{thm-alg-expconsequence} to $T_2^{[-1]}\bdot S$ with $X=\{1,k-2,k,k+1\}$ and $m=0$ to conclude that $0\in\Sigma_{\{1,(k-2),k,(k+1)\}\cdot q}(T)$, contrary to \eqref{tiltB}, which completes the proof.
\end{proof}


The means of transferring Propositions \ref{prop-small-d} and \ref{prop-large-d} into Theorems \ref{thm-sk-spec} and \ref{thm-sk-gen} is the following simple lemma.

\begin{lemma}
\label{lem-transfer}
Let $G$ be a finite abelian $p$-group with exponent $q$,  let $d=\left \lceil \frac{\mathsf D^*(G)}{q}\right\rceil$, and let $k_0\geq 1$. Suppose  $\mathsf s_{kq}(G)\leq kq+\mathsf D^*(G)-1$ for all $k\in[k_0,2k_0-1]$. Then $$\mathsf s_{kq}(G)\leq kq+\mathsf D^*(G)-1 \quad \mbox{ for all $k\geq k_0$}.$$
\end{lemma}

\begin{proof}
Let $k\geq k_0$ be arbitrary. Write $k=mk_0+r$ with $m\geq 0$ and $r\in[k_0,2k_0-1]$. Let $S$ be a sequence of terms from $G$  with $|S|=kq+\mathsf D^*(G)-1\geq mk_0q+\mathsf D^*(G)-1$. We need to show $0\in\Sigma_{kq}(S)$. By repeated application of the definition of $\mathsf s_{k_0q}(G)\leq k_0q+\mathsf D^*(G)-1$, we can find subsequences $T_1\bdot\ldots\bdot T_m\mid S$ such that each  $T_i$ is zero-sum with $|T_i|=k_0q$, for $i\in[1,m]$. But now $$|(T_1\bdot\ldots\bdot T_m)^{[-1]}\bdot S|=|S|-mk_0q=rq+\mathsf D^*(G)-1,$$ so applying the definition of $\mathsf s_{rq}(G)\leq rq+\mathsf D^*(G)-1$ to $(T_1\bdot\ldots\bdot T_m)^{[-1]}\bdot S$, we find another zero-sum subsequence $T_0\mid (T_1\bdot\ldots\bdot T_m)^{[-1]}\bdot S$ with $|T_0|=r\in[k_0,2k_0-1]$, and now $T=T_0\bdot T_1\bdot\ldots\bdot T_m$ is a zero-sum subsequence of $S$ with $|T|=(mk_0+r)q=kq$, as desired.
\end{proof}

We conclude with the proofs for both results regarding $\mathsf s_{k\exp(G)}(G)$.

\begin{proof}[Proof of Theorem \ref{thm-sk-spec}]
Let $k_0=d$. Since $p\geq 2d-1$, we have $p\geq k$ for every $k\in[k_0,2k_0-1]=[d,2d-1]$. Thus Proposition \ref{prop-small-d} implies that $\mathsf s_{kq}(G)\leq kq+\mathsf D^*(G)-1$ for every $k\in[k_0,2k_0-1]$, and the result now follows by applying Lemma \ref{lem-transfer}.
\end{proof}

\begin{proof}[Proof of Theorem \ref{thm-sk-gen}]
Let $k_0=\frac{d(d-1)}{2}+1$. Since $p\geq d^2-d+1$, we have $p\geq k$ for every $k\in[k_0,2k_0-1]=[\frac{d(d-1)}{2}+1,d^2-d+1]$. Thus Proposition \ref{prop-large-d} implies that $\mathsf s_{kq}(G)\leq kq+\mathsf D^*(G)-1$ for every $k\in[k_0,2k_0-1]$, and the result now follows by applying Lemma \ref{lem-transfer}.
\end{proof}

\end{document}